\documentclass[12pt]{amsart}

\usepackage[a-1b]{pdfx}   
\makeatletter \AtBeginDocument{\let\mathaccentV\AMS@mathaccentV} \makeatother

\usepackage{amsmath, amsthm, amssymb,mathrsfs}
\usepackage{fullpage}
\usepackage{color}
\usepackage[english]{babel}
\usepackage{hyperref}
\usepackage{soul}
\usepackage{enumitem}

\newtheorem{theorem}{Theorem}
\newtheorem{lemma}{Lemma}
\newtheorem{proposition}[lemma]{Proposition}

\newtheorem{definition}[lemma]{Definition}

\numberwithin{lemma}{section}

\newtheorem{manualtheoreminner}{Theorem}
\newenvironment{manualtheorem}[1]{%
  \renewcommand\themanualtheoreminner{#1}%
  \manualtheoreminner
}{\endmanualtheoreminner}

\numberwithin{equation}{section}

\newcommand{\bb}{\mathbb}
\newcommand{\abs}[1]{\left\vert { #1 } \right\vert}

\newcommand{\norm}[1]{\| #1 \|}
\begin{document}
\title{The Modified Energy Method for Quasilinear Wave Equations of Kirchhoff Type}

\author{ Ryan Martinez}
\address{Department of Mathematics, University of California at Berkeley}
\email{ryan\_martinez@berkeley.edu}

\begin{abstract}
    In this paper, we use the \emph{modified energy method} of Hunter, Ifrim, Tataru, and Wong
    to prove an improved quintic energy estimate for initial data small in $\dot H^1_x \times L^2_x$ for a wide class of 
    quasilinear wave equations of Kirchhoff type. This allows us to make the first steps towards small data 
    $H^{5/4}_x \times H^{1/4}_x$ local well-posedness. In particular, we prove an enhanced lifespan for corresponding solutions 
    depending only on the $\dot H^{5/4}_x \times \dot H^{1/4}_x$ norm of the initial data as well as the existence of weak 
    solutions for $H^{5/4}_x \times H^{1/4}_x$ initial data, again small in $\dot H^1_x \times L^2_x$. 
    In contrast to previous modified energy 
    results, the nonlinearity in these models depends on an $\dot H^1_x$ norm of the solution. This means 
    a modified energy cannot be deduced algebraically by analyzing resonant interactions between 
    wave packets since all spatial dependence is integrated out in the 
    nonlinearity. Instead, the modified energy is determined as a Taylor series of incremental 
    leading order terms.
\end{abstract}


\maketitle

\setcounter{tocdepth}{1}
\tableofcontents

\section{Introduction}
\label{sec:intro}
In this paper we study the quasilinear wave equation of Kirchhoff type
\begin{equation}
\label{eq:kirchhoff} \begin{cases}u'' = (1 + N(\norm{\nabla u}_{L_x^2}^2)) \Delta u \qquad \text{in } \bb R_t \times \bb R_x^n
\\u(0) = u_0 \in H_x^{s_0+1}(\bb R^n), u'(0) = u_1 \in H_x^{s_0}(\bb R^n) \end{cases}
\end{equation}
where $N: \bb R^+ \to \bb R$ is a $C^2_{loc}$ function with $N(0) = 0$. We work with 
initial data $(u_0, u_1) \in H^{s_0+1} \times H^{s_0}$ for $s_0 \geq 0$ and look for 
solutions in $C^0_t H_x^{s_0+1} \cap C_0^1 H_x^{s_0}$ which is typical for 
wave type equations. 

Note also that the choice to work in $\bb R^n$ is not essential. In fact, by the spectral theorem, the proofs follow
as written below for any Hilbert space $H$ and positive self-adjoint operator $-\Delta$ on $H$ with the nonlinearity replaced 
by 
$$N(\langle u, -\Delta u\rangle).$$

It is important that near small solutions \eqref{eq:kirchhoff} is uniformly of wave type. We make this restriction by factoring out the constant from $N$ and requiring that $N(0) = 0$.
We can see that our nonlinearity 
$N$ depends only on the $\dot H_x^1$ size of the data, 
and thus only depends on $t$ overall, since 
the norm is integrated over space. We will sometimes abuse notation and write $N(t)$ or 
simply $N$ for $N(\norm{\nabla u(t)}_{L^2_x}^2)$ when the context is clear.

One might hope for global well-posedness of \eqref{eq:kirchhoff} in the Sobolev space 
$\dot H_x^1 \times L_x^2$, which corresponds to $s_0 = 0$,
since this is the least regular space which controls the nonlinearity. 
However, as elaborated on below, local well-posedness in Sobolev spaces has not been improved
to below $H_x^{3/2} \times H_x^{1/2}$ ($s_0 = 1/2$), since its proof in the 90s.
On the other hand the global existence theory is more delicate with prior results requiring 
adding dampening to \eqref{eq:kirchhoff} or working in analytic spaces. 

The aim of this work is to take the first step towards improving the best known regularity for local well-posedness by proving energy estimates at the regularity of $\dot H_x^{5/4} \times \dot H_x^{1/4}$. As a consequence of these estimates, we are able to make improvements to the 
local and global theory. For the local theory, we show that for any $(u_0, u_1) \in H_x^{5/4} 
\times H_x^{1/4}$ which is small in $\dot H_x^1 \times L_x^2$, the Kirchhoff wave model has a weak solution 
in $C_t^0 (H_x^{5/4}) \cap W_t^{1, \infty} (H_x^{1/4})$ on some time interval. For the global existence theory, 
we show that
the time of existence for initial data $(u_0, u_1)$ small in $\dot H_x^1 \times L_x^2$ 
(including the weak solutions above)
depends only on the lower regularity Sobolev norm 
$$\epsilon = \norm{(u_0, u_1)}_{\dot H_x^{5/4} \times 
\dot H_x^{1/4}}$$
and that this time of existence is quintic, that is proportional to $\epsilon^{-4}$. 
Previously, this quintic dependence was only known on the torus and depending on the higher regularity 
inhomogeneous norm
$$\norm{(u_0, u_1)}_{H_x^{3/2} \times 
 H_x^{1/2}}.$$
Note that in the literature $\epsilon^{-4}$ dependence on norm size is called \emph{quintic} 
lifespan because it is the lifespan associated to an ODE with a quintic nonlinearity, namely 
$$\partial_t u = u^5.$$

\subsection{Related Work}

This model was proposed by Kirchhoff \cite{kirchhoff1897vorlesungen} to model 
waves where wave speed depends on how the wave deforms the medium, for instance in 
an elastic band. In particular, for the elastic band, the wave speed should depend 
only on the overall deformation of the material, which for waves is the potential energy
given by the square of the $\dot H_x^1$ norm of $u$.

The local well-posedness theory for the problem is explained in great detail in \cite{arosio1991mildly} and \cite{arosio1996well}
as well as the sources within. The main idea of the local well-posedness in \cite{arosio1991mildly}
and \cite{arosio1996well}
are very general and goes back to the work of Kato in \cite{kato1975cauchy}, but a more modern take on 
this general theory can be found in \cite{ifrim2022localwellposednessquasilinearproblems}. 
Note that in \cite{arosio1991mildly}, there is a focus on 
possible degeneracy of the non-linearity $1+N$, which they allow do decay to $0$ in a ``mild'' manner. 
While this is interesting, we will not focus on that here, opting instead for an analysis of 
the low regularity problem, which is studied in \cite{arosio1996well}. Using the 
techniques in \cite{arosio1996well} or \cite{ifrim2022localwellposednessquasilinearproblems},
\eqref{eq:kirchhoff} admits Hadamard local well-posedness for initial data 
$(u_0, u_1) \in H_x^{3/2} \times H_x^{1/2}.$ Hadamard local well-posedness refers to the existence
of a continuous, but not necessarily Lipschitz, map from initial data to a uniform in time space, 
in our case  $C_t^0(H_x^{3/2}) \cap C_t^1(H_x^{1/2}).$ The expectation to achieve only continuous dependence 
on initial data is named after Hadamard for his work for instance in \cite{hadamard1923lectures}.
Again, see \cite{ifrim2022localwellposednessquasilinearproblems} for a modern take on these ideas. 
Following \cite{ifrim2022localwellposednessquasilinearproblems}, 
the well-posedness in specifically $H_x^{3/2} \times H_x^{1/2}$ is possible because of two ideas. 
The first, which will not 
play a role in this paper, is the analysis of the linearized system (we explain why this 
analysis is troublesome at lower regularity in 
Section \ref{sec:lin}).
The second is that 
standard nonlinear energy estimates are cubic: putting 
$$E^s = \frac12 (1 + N(t)) \norm{\abs{\xi}^{1+s} \hat u}_{L_x^2}^2 
+ \frac12 \norm{\abs{\xi}^s \hat u'}_{L_x^2}^2$$
($E^s$ is the linearized energy coming from treating $N$ as a constant) we have that 
\begin{align}
    \nonumber\partial_t E^s &= \frac12 \partial_t N(t) \norm{\abs{\xi}^{1+s} \hat u}_{L^2}^2\\
&= N'(\norm{u}_{\dot H_x^1}^2)\int\abs{\xi_1}^{2+2s} \abs{\hat u(\xi_1)}^2 d\xi_1
\int \abs{\xi_2}^2 \Re(\hat u(\xi_2) \bar{\hat u}'(\xi_2)) d\xi_2 \label{eq:quad_energy} \\
                   \nonumber&\lesssim \norm{N'}_{L^\infty} E^s(t) \norm{u(t)}_{\dot H_x^{3/2}} 
                   \norm{u'(t)}_{\dot H_x^{1/2}} \lesssim \norm{N'}_{L^\infty} E^s E^{1/2}.
\end{align}
Then, setting $s = 1/2$, we can use Gronwall's inequality to get control of the 
$\dot H_x^{3/2} \times \dot H_x^{1/2}$ norm on a time scale inversely proportional to the initial 
energy. In particular if we denote the initial data size as 
$$\tilde \epsilon = \norm{(u_0, u_1)}_{\dot H_x^{3/2} \times \dot H_x^{1/2}},$$  
then we have control of $E^{1/2}$ on a time scale proportional to $\tilde \epsilon^{-2}.$
This in turn 
controls lower energies on the same time scale which is necessary to control 
the input to $N$, and thus results in the well-posedness in the inhomogeneous space 
$H_x^{3/2} \times H_x^{1/2}$. 

This work partially improves the result of \cite{arosio1996well} 
by proving the existence of weak solutions for initial data in the lower regularity Sobolev 
space $\dot H_x^{5/4} \times \dot H_x^{1/4}$. Note that this work's method is not strong enough 
to prove the uniqueness of such solutions or continuous dependence on initial data which is
discussed in Section \ref{sec:lin}.

Recently, the $\tilde \epsilon^{-2}$ time of existence has been improved in \cite{Baldi_2020} to $\tilde\epsilon^{-4}$ 
on the torus  with the special nonlinearity $N(\norm{\nabla u}_{L^2}^2) = \norm{\nabla u}_{L^2}^2$
and in \cite{Baldi_longer} to $\tilde\epsilon^{-6}$ under the same conditions 
for special ``non-resonant'' initial data. 
These works use the explicit spectral structure of the Laplacian 
on the torus as well as the more algebraic structure of the nonlinearity to do an explicit 
\emph{normal form transformation}, introduced by \cite{Shatah85}. Because of the inherent 
derivative loss of this method for quasilinear problems and complications from the spectrum of 
the Laplacian on the torus in higher dimensions, \cite{Baldi_2020} and \cite{Baldi_longer} 
require regularity 
$H_x^{3/2} \times H_x^{1/2}$ in one spatial dimension and $H_x^{2} \times H_x^{1}$ in higher 
spatial dimensions. 

This work directly improves the result of \cite{Baldi_2020} by lowering the regularity of 
the norm on which existence depends to $\dot H_x^{5/4} \times \dot H_x^{1/4}$ regardless of dimension 
and by generalizing the nonlinearity and ambient space significantly. Our results are 
stated in $\bb R^n$ for simplicity, but may be generalized to any Hilbert space and positive self-adjoint 
linear operator $-\Delta$ through the spectral theorem provided that Sobolev spaces are 
defined with respect to $\Delta$. This is because all of our work 
takes place on the frequency side and is agnostic to the specific structure of the spectrum of 
$\Delta$.
We will explain in more detail how this work's method relates to 
the method of normal forms used in \cite{Baldi_2020} and \cite{Baldi_longer} 
in Section \ref{ss:heuristics}.

We note that 
another well studied aspect of \eqref{eq:kirchhoff} is global well-posedness, which seems to 
require either extra dissipation in the equation or high regularity. With dissipation terms 
of one of the following forms:
$$-\partial_t u, \qquad -\Delta \partial_t u, \qquad \abs{\partial_t u}^\beta \partial_t u$$
\cite{ono1997global} and \cite{ONO19974449} study global existence and blow-up in the presence of 
an algebraic nonlinearity $\abs{u}^p u$, again with a more general nonlinearity which can mildly 
degenerate. These works get various threshold conditions for blowup and 
global existence based on initial energy. 

Another method in the literature for achieving global existence is by using initial data with
analytic structure using the Gevrey function spaces; see \cite{d1992global} and
\cite{nishihara1984global}.

\subsection{Results and Heuristics}\label{ss:heuristics}
As for our analysis of \eqref{eq:kirchhoff}, we will work exclusively on the frequency side, which 
enables the generalization of the results to any Hilbert space $H$ and positive self-adjoint operator $-\Delta$ on 
$H$ by the spectral theorem. We state
Theorems \ref{theorem:main} and \ref{theorem:enhanced} in $\bb R^n$ for simplicity only.
On the frequency side, the Sobolev norm $\dot H^s$ is defined by
$\norm{ \abs{\xi}^s \hat u}_{L^2}$
taken with respect to the spectral measure associated to $-\Delta$. By Plancherel's Theorem this agrees with the usual definition on domains $U$ with the Dirichlet Laplacian. 

In this work, the backbone of our results stem from improving the quadratic nonlinear energy 
estimate \eqref{eq:quad_energy} to a cubic estimate in the case of initial data small in 
$\dot H_x^1 \times L^2_x$. 

\begin{theorem}
    \label{theorem:main}
    Suppose that $u \in C_t^0([0,T], H_x^{s_0+1}) \cap C_t^1([0,T], H_x^{s_0})$, $s_0 \geq 1/4$ 
    solves \eqref{eq:kirchhoff}. Then there exists a family of modified energies 
    $E^s(t)$ with the property that for any $t \in [0, T]$ for which 
    $$ \norm{(u(t), u'(t))}_{\dot H_x^{1} \times L_x^2} \leq \delta \ll 1$$
    then for all $0 \leq s \leq s_0$ we have the comparability bound
    $$\norm{(u(t), u'(t))}^2_{\dot H_x^{1+s} \times \dot H_x^{s}} \sim_{s_0} E^{s}(t)$$
    and the quintic energy estimate 
    $$\partial_t E^{s}(t) \lesssim_{s_0} E^{s}(t) (E^{1/4}(t))^2.$$
\end{theorem}

Note that the constants will depend on the nonlinearity $N$, but we suppress the dependence for brevity.
There is no obstruction to proving the above theorem for $s < 0$, but we exclude this case
since it is not required for the following results. However, one might desired Theorem \ref{theorem:main} for certain 
$s < 0$ in a proof of lower regularity local well-posedness, which is discussed in Section \ref{sec:lin}.

A corollary of Theorem \ref{theorem:main} and the local well-posedness 
in $H_x^{s+1} \times H_x^{s}$ for $s \geq 1/2$ is an enhanced lifespan for
initial data small in 
$\dot H_x^1 \times L^2_x$. 

\begin{theorem}
    \label{theorem:enhanced}
    Let $\delta$ be the small constant in Theorem \ref{theorem:main}.
    Suppose that the initial data $(u_0, u_1) \in H_x^{s_0+1} \times H_x^{s_0},$ $s_0 \geq 1/2$ has 
    $$\norm{(u_0, u_1)}_{\dot H_x^{5/4} \times \dot H_x^{1/4}} = \epsilon$$ 
    and
    $$\norm{(u_0, u_1)}_{\dot H_x^{1} \times L_x^2} = \delta_0 \ll_{s_0, \epsilon} \delta.$$
    Then 
    there exists a unique solution 
    $$u \in C_t^0([0,T], H_x^{s+1}) \cap C_t^1([0,T],H_x^{s})$$ 
    to \eqref{eq:kirchhoff} where $T \sim_{s_0} 1/\epsilon^4.$
    Furthermore, for all $0 \leq s \leq s_0$, we have the energy estimates on $[0,T]$
    $$\sup_{0 \leq t \leq T}\norm{(u, \partial_t u)}_{\dot H_x^{s+1} \times \dot H_x^{s}} 
    \lesssim_{s_0, \epsilon} \norm{(u_0, u_1)}_{\dot H_x^{s+1} \times \dot H_x^s}.$$
    In particular
    $$\sup_{0 \leq t \leq T}\norm{(u, \partial_t u)}_{\dot H_x^{1} \times L_x^2} \leq \delta$$
    and
    $$\sup_{0 \leq t \leq T}\norm{(u, \partial_t u)}_{\dot H_x^{5/4} \times \dot H_x^{1/4}} \lesssim \epsilon.$$
\end{theorem}

Note that this extended lifespan is a significant step towards 
local well-posedness in the lower regularity space $H_x^{5/4} \times H_x^{1/4}.$ In fact, a direct 
consequence of this result is the existence of weak solutions for compactly 
supported $H_x^{5/4} \times H_x^{1/4}$ initial
data, which is given in the next theorem. The reason that we require compactly supported initial 
data is that the result of Theorem \ref{theorem:enhanced} is only enough to give solutions as weak limits. 
In particular, to guarantee that these weak limits are solutions we need strong control over 
$\norm{\nabla u(t)}_{C^0_t(L^2_x)},$ which requires some compactness. 

This constraint is reflective of the fact that Theorem \ref{theorem:main} is not strong enough to 
ensure that solutions are unique limits of smooth solutions, which is why we call the solutions constructed 
in the following theorem \emph{weak solutions}. In Section \ref{sec:lin}, we discuss why Theorem \ref{theorem:main} cannot be used to prove that solutions are unique. 

\begin{theorem}\label{thm:weak_solutions}
    Let $(u_0,u_1) \in H_x^{5/4} \times H_x^{1/4}$ be compactly supported with  
    $$\norm{(u_0, u_1)}_{\dot H_x^1 \times L_x^2} = \delta_0 \ll \delta$$
    where $\delta$ is the same small constant as in Theorem \ref{theorem:enhanced}.
    Denote the size of the initial data in $\dot H_x^{5/4} \times \dot H_x^{1/4}$ by 
    $$\epsilon = \norm{(u_0, u_1)}_{\dot H_x^{5/4} \times \dot H_x^{1/4}}.$$
    Then, for $T \sim 1/\epsilon^4$, as in Theorem \ref{theorem:enhanced}, 
    there exists a weak solution 
    $$u \in C_t^0([0,T], H_x^{5/4})\cap C^{1}_t([0,T], H_x^{1/4})$$
    which solves \eqref{eq:kirchhoff} for times in $[0,T)$ in the sense of distributions:

    For all $\phi \in C^\infty_c([0,T) \times U)$ (which vanish at $t=T$) we have 
    $$\int_{[0,T]\times \bb R^n} u \partial_t^2 \phi - u \Delta \phi  - u N(\norm{\nabla u}^2_{L^2})\Delta \phi \, dt dx
    = \int_{\bb R^n} u_0 \partial_t \phi \, dx - \int_{\bb R^n} u_1 \phi \, dx.$$
\end{theorem}

We call the energies $E^s$ in Theorem \ref{theorem:main} \emph{modified energies} for this problem. 
We will see the exact form that $E^s$ takes in Definition \ref{def:model_energy} 
in a model case of the  nonlinearity $N$ and 
in Definition \ref{def:full_energy} in the general case of the nonlinearity $N$. 

It is the main challenge of this work to construct $E^s$, and the purpose of the following discussion is to explain the motivation for and the difficulty of the construction.

Ultimately, a modified energy takes advantage of special time oscillations in the standard energy 
estimate \eqref{eq:quad_energy}. Having time oscillations suggests that we can improve the 
the integral of \eqref{eq:quad_energy} by taking advantage of cancellation in the integral.
The idea to capture these oscillations via modified energies is due 
to \cite{HITW13}, and improves over simply integrating by parts in quasilinear problems 
by making clear which part of the time oscillations can be absorbed into the original energy. 

In contrast to \eqref{eq:kirchhoff}, in \cite{HITW13}, nonlinearities are algebraic in nature, 
and so time
oscillations can be fully described by interactions with respect to spatial frequencies. In particular, 
this allows the authors in \cite{HITW13} to deduce a change of variables, or 
\emph{normal form transformation}. See \cite{Shatah85} for the use of normal forms in semilinear problems. The normal form transformation is unbounded for quasilinear problems 
and cannot be used directly. However, \cite{HITW13} shows how to use a normal form transformation to construct an associated \emph{modified energy}. 
 
For quasilinear problems, this method has seen much success. For their use in local well-posedness 
see \cite{hunter2014dimensionalwaterwavesholomorphic} and for their use in enhanced lifespan see \cite{ifrim2017lifespansmalldatasolutions}. 

The nonlinearity in \eqref{eq:kirchhoff} is not algebraic and in fact integrates out the spatial 
variable completely, which makes finding a normal form directly hopeless. The novel idea of this 
work is to derive a modified energy in the following two steps. First, we find the leading order part of the modified energy
by looking at the algebraic structure of 
\eqref{eq:quad_energy}, which is akin to the division problem in the standard normal form method. To complete the construction, we interpret the remaining part as the sum of 
perturbative terms and a leading order non-perturbative term, which is, up to a twist, the product of \eqref{eq:quad_energy}, and a low frequency perturbative term. We can then derive the modified energy as a Taylor series, again with a twist, in this low frequency perturbative term.
This idea is akin to an exponential conjugation 
present in \cite{hunter2014dimensionalwaterwavesholomorphic}. In that work, a transformation 
$$\tilde u = e^{\phi(u)} u$$ is used. Considering the Taylor expansion of the exponential, this 
corresponds to a sequence of polynomial corrections to $u$. The relationship between 
this exponential conjugation and a series of corrections was noted in 
\cite{ai2023dimensionalgravitywaveslow}, but, to our knowledge, this work is the first time 
a series of corrections was required because of the absence of a clear algebraic normal form 
transformation. In particular, it is not clear that the resulting Taylor expansion 
in this work can be expressed as the energy after a simple change of variables because of 
its unique dependence on the nonlinearity $N$ as well as lower frequency parts of the solution $u$. 
The expansion is reflected in Definition \ref{def:f} in a model case for the 
nonlinearity $N$ and Definition \ref{def:gen_f}
in the full case.

Before getting into the proof, 
we present a heuristic for why 
we expect time oscillations to be present in \eqref{eq:quad_energy}.

~

We will explain this heuristic in the simplest setting where the nonlinearity is given by
$$N(\norm{\nabla u}_{L^2}^2) = \norm{\nabla u}_{L^2}^2.$$
First we decompose a solution $u$ to \eqref{eq:kirchhoff} into dyadic spatial frequency 
pieces 
$$u = \sum u_\lambda$$
where $\hat u_\lambda$ is localized near frequencies of size $\lambda$. At least heuristically, 
this diagonalizes \eqref{eq:kirchhoff} to 
$$u''_\lambda - \lambda^2(1 + \sum_{\mu} \mu^2 \abs{u_\mu}^2) u_\lambda = 0.$$
As long as the $\dot H_x^1$ norm of $u$ is small, we expect $u_\lambda$ to oscillate in time 
at frequency approximately $\lambda$. Being more precise, for small times we do not expect 
the $\dot H_x^1$ norm of $u$ to change very much and so to first order in $t$, the time frequency, which 
we will denote $\phi_\lambda(t)$, should be 
$$\phi_\lambda(t) = \lambda t \sqrt{1 + \norm{u_0}_{\dot H_x^1}^2} + \mathcal O(t^2).$$
Then we can heuristically write 
$$\hat u_\lambda(t) \approx c^+_\lambda e^{i\phi_\lambda(t)} + c^-_\lambda e^{-i\phi_\lambda(t)}$$
where $c^{\pm}_\lambda$ depends on $(u_0)_\lambda$ and $(u_1)_\lambda$. 
(We pretend that $c$ is constant when in reality it varies over frequency regions, all this 
does is simplify the notation below, which is heuristic anyway).
Using this heuristic decomposition we look at the energy estimate terms \eqref{eq:quad_energy}.
We see 
\begin{align*}
    \int\abs{\xi_1}^{2+2s} \abs{\hat u(\xi_1)}^2 d\xi_1
    &\approx 
    \sum_{\lambda_1} \lambda_1^{2+2s} (c^+_{\lambda_1} e^{i\phi_{\lambda_1}} 
        + c^-_{\lambda_1} e^{-i\phi_{\lambda_1}})(\bar c^+_{\lambda_1} e^{-i\phi_{\lambda_1}} 
        + \bar c^-_{\lambda_1} e^{i\phi_{\lambda_1}})\\
    &=
    \sum_{\lambda_1} \lambda_1^{2+2s} (\abs{c^+_{\lambda_1}}^2  + \abs{c^-_{\lambda_1}}^2 
    + 2\Re(c^+_{\lambda_1}\bar c^-_{\lambda_1} e^{2i\phi_{\lambda_1}}))
\end{align*}
and 
\begin{align*}
    \int\abs{\xi_2}^{2} \Re(\hat u(\xi_2) \bar{\hat u}'(\xi_2)) d\xi_2
    &\approx 
    \sum_{\lambda_2} \lambda_2^{2} \Re((c^+_{\lambda_2} e^{i\phi_{\lambda_2}} 
        + c^-_{\lambda_2} e^{-i\phi_{\lambda_2}})
        (-i\phi_{\lambda_2}'\bar c^+_{\lambda_2} e^{-i\phi_{\lambda_2}} 
        + i\bar c^-_{\lambda_2}\phi_{\lambda_2}' e^{i\phi_{\lambda_2}}))\\
    &=
    \sum_{\lambda_2} \lambda_2^2 \phi_{\lambda_2}'\Re(i \abs{c^+_{\lambda_2}}^2  + i\abs{c^-_{\lambda_2}}^2)
    + 2\Im(c^+_{\lambda_2}\bar c^-_{\lambda_2} e^{2i\phi_{\lambda_2}})\\
    &= \sum_{\lambda_2} 2\lambda_2^2 \phi_{\lambda_2}' \Im(c^+_{\lambda_2}\bar c^-_{\lambda_2} e^{2i\phi_{\lambda_2}}).
\end{align*}

Thus, since $\phi_\lambda \approx C \lambda t$ 
the product of these two terms can only have no oscillations when $\lambda_1 = \lambda_2$ coming 
from the product of the oscillatory pieces or when $\lambda_2 =0$. We see that when $\lambda_2 = 0$, 
the factor of $\phi_{\lambda_2}'$ is 0 to first order and so, to leading order this will not 
produce a non-oscillatory part. Similarly when $\lambda_1 = \lambda_2 \neq 0$
\begin{align*}
    \lambda^{2+2s}  &2\Re(c^+_{\lambda}\bar c^-_{\lambda} e^{2i\phi_{\lambda}}))
2\lambda^2 \phi_{\lambda}' \Im(c^+_{\lambda}\bar c^-_{\lambda} e^{2i\phi_{\lambda}})\\
&= \lambda^{4+2s} (c^+_\lambda \bar c^-_\lambda e^{2i\phi_\lambda} + \bar c^+_\lambda c^-_\lambda e^{-2i\phi_\lambda})(c^+_\lambda \bar c^-_\lambda e^{2i\phi_\lambda} - \bar c^+_\lambda c^-_\lambda e^{-2i\phi_\lambda})\\ 
&= \lambda^{4+2s} (c^+_\lambda \bar c^-_\lambda)^2 e^{4i\phi_\lambda} - 
(\bar c^+_\lambda c^-_\lambda)^2 e^{-4i\phi_\lambda} 
\end{align*}
which oscillates! Thus, at least heuristically, we expect to be able to add a term to our energy 
$E^s$ whose time derivative cancels \eqref{eq:quad_energy} to leading order.

\subsection{Outline}

The outline of the paper will be as follows: first in Section 
\ref{sec:simple} we discuss the modified energies in Theorem \ref{theorem:main} in the 
model case where $N$ is quadratic in the $\dot H_x^1$ norm of $u$. The bulk of the interesting 
ideas are present here which is why we separate them out. In Section \ref{gen}
we consider the general case of $N \in C^2_{loc}$ which adds only minor details to the 
model case. In Section \ref{sec:implications}, we discuss the implications of Theorem \ref{theorem:main} to the well-posedness 
theory for \eqref{eq:kirchhoff} by giving the proofs of Theorems \ref{theorem:enhanced} and \ref{thm:weak_solutions}. 
In Section \ref{sec:lin}, we discuss how the modified energies in Theorem \ref{theorem:main} fall short of fully
proving local well-posedness at the level of $H_x^{5/4} \times H_x^{1/4}.$ Briefly, the reason this 
method fails for \eqref{eq:kirchhoff} is because the method described in Section \ref{sec:simple} 
cannot be applied to the linearization of \eqref{eq:kirchhoff} around a solution. We show why this 
is using the heuristic presented above as well as by demonstrating exactly where the method in 
Section \ref{sec:simple} fails. Finally, there are several more technical proofs which
would distract from the presentation. These proofs are given in the appendix: Section \ref{sec:appendix}.

\subsection{Notation }
We reserve the notation $u'$ for the time derivative of $u$ and $\hat u$ for the spatial Fourier transform of $u$. We often suppress the dependence of $\hat u$ and $\hat u'$ on $t$ since the dependence on $\xi$ is often more important. 

We will often refer to the quantity 
$$\norm{(u(t), u'(t))}_{\dot H^{s+1}_x \times \dot H^s_x}$$
at some time $t$ as simply 
``the size of a solution $u$ in $\dot H^{s+1}_x \times \dot H^s_x$''
when the context is clear. Similarly, when we refer to ``the initial data'' we are referring to the pair 
$$(u_0, u_1) \in H^{s+1}_x \times H^s_x.$$

In Sections \ref{sec:simple} and \ref{gen} we will benefit from having a shorthand notation for the minimum 
length of two frequency vectors. We use the notation 
$$\xi_1 \wedge \xi_2 := \min(\abs{\xi_1}, \abs{\xi_2}).$$
Further several integrals will depend on several frequencies $\xi_1 \ldots \xi_m$. We use the 
notation $\vec{d\xi}$ to denote that we are integrating with respect to all such variables.

Also, note that in what follows, we assume all integrals converge for our solutions $u$. 
This is easily justified by proving all results for solutions $u$ with initial data in the 
Schwartz space $\mathscr S$ and then using the continuity of the solution map for $H^{3/2} \times H^{1/2}$ initial data to 
infer the estimates on Sobolev spaces by a density argument.

\subsection{Acknowledgments}
The author was supported by the NSF grant DMS-2054975 as well as by a Simons Investigator grant
from the Simons Foundation. Some of this work was carried out while the author was 
participating in the Erwin Schr\"odinger Institute program ``Nonlinear Waves and Relativity" during
Summer 2024. Finally, the author would like to acknowledge Daniel Tataru whose advising made 
this work possible.
\section{The Model Case}
\label{sec:simple}
In this section, we will calculate the modified energy in the model case where 
$$N(\norm{\nabla u}_{L^2}^2) = A\norm{\nabla u}_{L^2}^2$$ 
for some constant $A$ which may have either sign.
The model case will form the basis of the more general case in Section \ref{gen}.

It will be helpful to introduce the following notation for the unmodified energy density at 
frequency $\xi_1$.
\begin{definition} \label{def:unmodified}
Let $u \in C_t^0([0,T], H_x^{s_0+1}) \cap C_t^1([0,T], H_x^{s_0})$ solve \eqref{eq:kirchhoff} in the model case 
with $s_0 \geq 0$. We define the \emph{unmodified energy density at regularity $0 \leq s \leq s_0$} 
as 
$$E^s_{\xi_1} = \frac12 \left(1 + A\int \abs{\xi_2}^2 \abs{\hat u(\xi_2)}^2 d\xi_2\right)
\abs{\xi_1}^{2+2s} \abs{\hat u(\xi_1)}^2
+\frac12 \abs{\xi_1}^{2s} \abs{\hat u'(\xi_1)}^2.$$
Note that the unmodified energy can be recovered from this density by integrating over all frequencies
$$\int E^s_{\xi_1} d\xi_1.$$
\end{definition}
We use this notation to emphasize that this energy is first order in the sense that it 
depends on $u$ through only one frequency, which distinguishes it 
from higher order corrections that we will add later on. In the following subsections,
we will look at how the energy correction arises in this model case. The modified energy is 
the model case is given in Definition \ref{def:model_energy}.

\subsection{Derivation of the Leading Order Correction}\label{ssec:derivation}

A simple calculation shows that 
\begin{equation}\label{eq:init} \partial_t E^s_{\xi_1} = \int_{\xi_2} A 
\abs{\xi_1}^{2+2s}\abs{\xi_2}^2 \abs{\hat u(\xi_1)}^2 
\Re(\hat u(\xi_2) \bar {\hat u}'(\xi_2)) d\xi_2.
\end{equation}
In order to improve our estimates to cubic estimates, we want to 
cancel the above quadratic remainder by using terms of the form
$\hat u(\xi_1)^2 \hat u(\xi_2)^2$. To leading order, the correction is determined
algebraically. Keeping our convention, we will use the notation 
$E^s_{\xi_1\xi_2}$ to denote this leading order correction, which is second order in the 
sense that it depends on $u$ through two frequencies. We define the second order energy correction 
as follows.
\begin{definition} \label{def:coefficients}
Let $u \in C_t^0([0,T], H_x^{s_0+1}) \cap C_t^1([0,T], H_x^{s_0})$ solve \eqref{eq:kirchhoff} in the model 
case
with $s_0 \geq 0$. We define the \emph{second order energy density at regularity $0 \leq s \leq s_0$} 
as 
$$E^s_{\xi_1\xi_2} = a_{\xi_1\xi_2} \abs{\hat u(\xi_1)}^2 \abs{\hat u(\xi_2)}^2
+ b_{\xi_1\xi_2} \abs{\hat u(\xi_1)}^2\abs{\hat u'(\xi_2)}^2
+ c_{\xi_1\xi_2} \Re(\hat u(\xi_1) \bar {\hat u}'(\xi_1))
\Re( \hat u(\xi_2) \bar {\hat u}'(\xi_2)) $$
where 
$$a_{\xi_1\xi_2} = -\frac18 \abs{\xi_1}^2 \abs{\xi_2}^2(\abs{\xi_1}^{2s} + \abs{\xi_2}^{2s})$$
$$b_{\xi_1\xi_2} = -c_{\xi_1\xi_2} = -\frac14 \abs{\xi_1}^2 \abs{\xi_2}^2 
\frac{\abs{\xi_1}^{2s} - \abs{\xi_2}^{2s}}{\abs{\xi_1}^2 - \abs{\xi_2}^2}.$$
\end{definition}

That this cancels the remainder terms coming from the first order energy density $E^s_{\xi_1}$
is given by the following proposition.

\begin{proposition}
    \label{prop:coefficients}
    Let $u \in C_t^0([0,T], H_x^{s_0+1}) \cap C_t^1([0,T], H_x^{s_0})$ solve \eqref{eq:kirchhoff} in 
    the model case for some 
    $s_0 \geq 0$ and let 
    $0 \leq s \leq s_0$.
Then 
\begin{align*}
    \partial_t \int_{\xi_1, \xi_2} A E^s_{\xi_1\xi_2}&\vec{d\xi}  = 
    -\int_{\xi_1, \xi_2} A\abs{\xi_1}^{2+2s}\abs{\xi_2}^2 \abs{\hat u(\xi_1)}^2 
\Re(\hat u(\xi_2) \bar {\hat u}(\xi_2))\vec{d\xi} \\
&- \int_{\xi_1,\xi_2,\xi_3} \frac12 A^2 
(\abs{\xi_1}^{2s} - \abs{\xi_2}^{2s})\abs{\xi_1}^2\abs{\xi_2}^2\abs{\xi_3}^2
\abs{\hat u(\xi_1)}^2 \Re(\hat u(\xi_2) \bar {\hat u}'(\xi_2)) \abs{\hat u(\xi_3)}^2\vec{d\xi}. 
\end{align*}
\end{proposition}
The proof is a straightforward computation and can be found in Subsection \ref{app:coeff} of the appendix. Of note, however, is the heavy use of the symmetry coming from 
integrating over all $\xi_1, \xi_2$:
$$\int_{\xi_1, \xi_2} F(\xi_1, \xi_2) \vec{d\xi} = 
\int_{\xi_1, \xi_2} F(\xi_2, \xi_1) \vec{d\xi} = 
\frac12 \int_{\xi_1, \xi_2} F(\xi_1, \xi_2) + F(\xi_2, \xi_1)\vec{d\xi}. $$
In the statement of the proposition, we separate the constant $A$ from our energy densities.
This will be useful when we extend to arbitrary $C^2$ functions $N$ in Section \ref{gen},
where $A$ will no longer be a constant.

In order for our modified energy to control the $\dot H_x^{s+1} \times \dot H_x^s$ norm of $u$, 
it is important for the correction to be dominated by the unmodified energy. We will only be able to 
get this control by assuming that $u$ is small in the $\dot H_x^1 \times L_x^2$ norm, which 
justifies this assumption in Theorem \ref{theorem:main}. To show that this is true
we need control over the coefficients from 
Definition \ref{def:coefficients}. 

Control of the nontrivial coefficient $b_{\xi_1\xi_2}$ 
is given by the following algebraic lemma whose proof we leave to Subsection \ref{app:kernel} of the appendix.
\begin{lemma}
    \label{lem:kernel}
    When $\abs{\xi_1} \leq \abs{\xi_2}$ and $s \geq 0$ we have
    $$\abs{\frac{\abs{\xi_1}^{2s} - \abs{\xi_2}^{2s}}{\abs{\xi_1}^2 - \abs{\xi_2}^2}}
    \leq (1+s) \frac{\abs{\xi_2}^{2s}}{\abs{\xi_2}^2}.$$
    When $\abs{\xi_1} \leq \abs{\xi_2}$ and $s \leq 0$ we have
    $$\abs{\frac{\abs{\xi_1}^{2s} - \abs{\xi_2}^{2s}}{\abs{\xi_1}^2 - \abs{\xi_2}^2}}
    \leq (1+\abs{s}) \frac{\abs{\xi_1}^{2s}}{\abs{\xi_2}^2}.$$
\end{lemma}
We note that there is no obstruction to proving Theorem \ref{theorem:main} for $s < 0$ using 
the second part of the above lemma. The reason one might desire Theorem \ref{theorem:main} for 
$s < 0$ is to form a contraction in a lower regularity space which is enough for 
Hadamard well-posedness. 
See \cite{ifrim2022localwellposednessquasilinearproblems} for details.
However, in this work we do not do this analysis for reasons described in 
Section \ref{sec:lin}.

Using Lemma \ref{lem:kernel} we get the following control over our energy correction.
\begin{proposition} 
\label{prop:quad_bounds}
Let $u \in C_t^0([0,T], H_x^{s_0+1}) \cap C_t^1([0,T], H_x^{s_0})$ solve \eqref{eq:kirchhoff} in the model 
case with $s_0 \geq 0$. 
Then
$$\abs{\int_{\xi_1,\xi_2} E^s_{\xi_1\xi_2} \vec{d\xi} } \lesssim (1+s)
\norm{(u,u')}^2_{\dot H_x^{s+1}\times \dot H_x^s}
\norm{(u,u')}^2_{\dot H_x^1 \times L_x^2}$$ 
where $E^s_{\xi_1\xi_2}$ is defined in Definition \ref{def:coefficients}.
\end{proposition}
\begin{proof}
For the $a$ term we simply apply Fubini's theorem
$$\abs{\int_{\xi_1, \xi_2} a_{\xi_1\xi_2} \abs{\hat u(\xi_1)}^2 \abs{\hat u(\xi_2)}^2\vec{d\xi} }
\leq \frac14 \left(\int_{\xi_1}  \abs{\xi_1}^{2+2s} \abs{\hat u(\xi_1)}^2 d\xi_1\right)
\left(\int_{\xi_2}  \abs{\xi_2}^{2} \abs{\hat u(\xi_2)}^2 d\xi_2\right).$$
For the $b$ and $c$ terms, we need to analyze their kernel using Lemma \ref{lem:kernel}. We start 
with $b$.

When $\abs{\xi_1} \leq \abs{\xi_2}$ we see that
$$\abs{\xi_1}^2\abs{\xi_2}^2
\abs{\frac{\abs{\xi_1}^{2s} - \abs{\xi_2}^{2s}}{\abs{\xi_1}^2 - \abs{\xi_2}^2}}
\leq (1+s) \abs{\xi_1}^2\abs{\xi_2}^{2s}$$
which gives us the estimate
$$\abs{\int_{\xi_1 \leq \xi_2} b_{\xi_1\xi_2} \abs{\hat u(\xi_1)}^2 
\abs{\hat u'(\xi_2)}^2\vec{d\xi} }
\leq \frac14 (1+s) \left(\int_{\xi_1}  \abs{\xi_1}^{2} \abs{\hat u(\xi_1)}^2 d\xi_1\right)
\left(\int_{\xi_2}  \abs{\xi_2}^{2s} \abs{\hat u'(\xi_2)}^2 d\xi_2\right).$$
And when $\abs{\xi_2} \leq \abs{\xi_1}$ we see that
$$\abs{\xi_1}^2\abs{\xi_2}^2
\abs{\frac{\abs{\xi_1}^{2s} - \abs{\xi_2}^{2s}}{\abs{\xi_1}^2 - \abs{\xi_2}^2}}
\leq (1+s) \abs{\xi_1}^{2s}\abs{\xi_2}^{2} \leq (1+s) \abs{\xi_1}^{2s+2}$$
which gives us the estimate
$$\abs{\int_{\xi_1 \geq \xi_2} b_{\xi_1\xi_2} \abs{\hat u(\xi_1)}^2 
\abs{\hat u'(\xi_2)}^2\vec{d\xi} }
\leq \frac14 (1+s) \left(\int_{\xi_1}  \abs{\xi_1}^{2+2s} \abs{\hat u(\xi_1)}^2 d\xi_1\right)
\left(\int_{\xi_2}  \abs{\hat u'(\xi_2)}^2 d\xi_2\right).$$

For the $c$ term we can use symmetry to assume $\abs{\xi_1} \leq \abs{\xi_2}$. Then
$$\abs{\xi_1}^2\abs{\xi_2}^2
\abs{\frac{\abs{\xi_1}^{2s} - \abs{\xi_2}^{2s}}{\abs{\xi_1}^2 - \abs{\xi_2}^2}}
\leq (1+s) \abs{\xi_1}^{2}\abs{\xi_2}^{2s} 
\leq (1+s) \abs{\xi_1}^{1}\abs{\xi_2}^{2s+1}.$$
Then we get 
\begin{align*}
    \Bigg\vert\int_{\abs{\xi_1} \leq \abs{\xi_2}}&
c_{\xi_1\xi_2} \Re(\hat u(\xi_1) \bar {\hat u}'(\xi_1))
\Re( \hat u(\xi_2) \bar {\hat u}'(\xi_2)) \vec{d\xi} \Bigg\vert\\
&\leq \frac14 (1+s)
\int_{\xi_1}
\abs{\xi_1}^1\abs{\hat u(\xi_1)} \abs{\hat u'(\xi_1)}d\xi_1
\int_{\xi_2}
\abs{\xi_2}^{s+1}\abs{\hat u(\xi_2)} \abs{\xi_2}^s \abs{\hat u'(\xi_2)}d\xi_2
\end{align*}
which is controlled by Cauchy-Schwarz.

\end{proof}

\subsection{Necessity of Higher Order Corrections}
\label{ssec:necessity}

Now that we know our modification is controlled, all that remains is to control the 
time derivative of modified energy as in Theorem \ref{theorem:main}. In the remainder 
of this subsection we will give a brief description of how we \emph{fail} to get control 
using our second order correction for certain unfavorable balances of frequencies,
and how this prompts us to look at higher order corrections.

First, let us look at the exact form the derivative takes.
\begin{align}
    \partial_t &\left(\int_{\xi_1} E^s_{\xi_1}d\xi_1 + \int_{\xi_1,\xi_2} A E^s_{\xi_1\xi_2}\vec{d\xi}\right)\nonumber\\
& \label{eq:first_order}
= - \int_{\xi_1,\xi_2,\xi_3} \frac12 A^2 
(\abs{\xi_1}^{2s} - \abs{\xi_2}^{2s})\abs{\xi_1}^2\abs{\xi_2}^2\abs{\xi_3}^2
\abs{\hat u(\xi_1)}^2 \Re(\hat u(\xi_2) \bar {\hat u}'(\xi_2)) \abs{\hat u(\xi_3)}^2 \vec{d\xi}
\end{align}

To see that there are some balances of frequencies where we \emph{do} get control, we consider 
the best possible case: when all frequencies are comparable.
Here \eqref{eq:first_order} is controlled by
$$\int_{\xi_1 \sim \xi_2 \sim \xi_3 \sim \xi}\abs{\xi}^{6+2s} 
\abs{\hat u(\xi_1)}^2 \abs{\hat u(\xi_2)\hat u'(\xi_2)}\abs{\hat u(\xi_3)}^2\vec{d\xi}.$$ 
We can try to split up the $\xi^{6+2s}$ factor among the factors of $u$ to put each 
$u$ in $\dot H_x^{1+s}$ and 
$u'$ in $\dot H_x^s$ (using Cauchy-Schwarz for the $\xi_2$ terms), which is optimal when $6+2s = 5(1+s) + s$ or $s = 1/4$. This is what 
justifies the space $\dot H_x^{5/4} \times \dot H_x^{1/4}$ in Theorem \ref{theorem:main}.

However, we can see that in balances where $\xi_2$ is larger than $\xi_3$ even the 
$\abs{\xi_1}^{2s}$ term in \eqref{eq:first_order} will give us trouble because 
we cannot control 
$$\int_{\xi_2}\abs{\xi_2}^2 \abs{\hat u(\xi_2)\hat u'(\xi_2)}d\xi_2$$
in $\dot H_x^{5/4} \times \dot H_x^{1/4}$ and have nowhere to move the extra factors of $\xi_2$.

Fortunately, the terms we cannot control can be written as a low frequency 
term (which can be treated like a constant) times the very same quadratic energy terms 
in \eqref{eq:init} which
we cancelled! We illustrate this explicitly when $\abs{\xi_3} \leq \abs{\xi_1}, \abs{\xi_2}.$

First, we eliminate the $\abs{\xi_2}^{2s}$ term by ``integrating by parts'' in time. In 
other words, we add to our second order energy correction the expression
$$-\int_{\abs{\xi_3} \leq \abs{\xi_1}, \abs{\xi_2}} \frac14 A^2 
\abs{\xi_2}^{2s}\abs{\xi_1}^2\abs{\xi_2}^2\abs{\xi_3}^2
\abs{\hat u(\xi_1)}^2 \abs{\hat u(\xi_2)}^2 \abs{\hat u(\xi_3)}^2\vec{d\xi}.$$
This changes \eqref{eq:first_order} by canceling the $\abs{\xi_2}^{2s}$ term 
when the time derivative lands on $\abs{\hat u(\xi_2)}^2$; it doubles contribution of 
$\abs{\xi_1}^{2s}$ when the time derivative lands on $\abs{\hat u(\xi_1)}^2$ (after 
relabelling $\xi_1 \leftrightarrow \xi_2$); and it adds a term we can control when 
the derivative lands on $\abs{\hat u(\xi_3)}^2$. We call this ``integrating by parts'' 
because we have moved the time derivative onto more favorable terms using the product rule.

The reason we said earlier that we can treat the 
low frequency factor $\abs{\hat u(\xi_3)}^2$ as a constant is \emph{because} we will be able 
to control 
terms where the time derivative lands on this factor! This will be a theme going forward, so 
keep this in mind.

Going back to our analysis, after integrating by parts we are left with 
$$- \int_{\xi_3\leq \xi_1,\xi_2} A^2 \left(\abs{\xi_3}^2 \abs{\hat u(\xi_3)}^2\right)
\abs{\xi_1}^{2s}\abs{\xi_1}^2\abs{\xi_2}^2
\abs{\hat u(\xi_1)}^2 \Re(\hat u(\xi_2) \bar {\hat u}'(\xi_2))\vec{d\xi} .
$$
Indeed, this is exactly \eqref{eq:init} but with a different coefficient:
$$\int_{\xi_1, \xi_2} -\left(A^2 \int_{\xi_3 \leq \xi_1, \xi_2}\abs{\xi_3}^2 
\abs{\hat u(\xi_3)}^2d\xi_3\right) E^s_{\xi_1 \xi_2} \vec{d\xi}. 
$$
Note that the coefficient is symmetric in $\xi_1$ and $\xi_2$ so all the manipulations 
we did in Proposition \ref{prop:coefficients} will carry through again, and when the derivative 
falls on $\hat u(\xi_3)$ we should still have control since $\xi_3$ is small compared 
to $\xi_1$ and $\xi_2$.

Of course, this correction will generate another term coming from the
nonlinearity that needs correction! 
Before we see how we can fold all of these corrections together, we note that 
we will need a similar higher order correction in the case where $\abs{\xi_1} \leq 
\abs{\xi_2}, \abs{\xi_3}$. In this situation, we can control the $\abs{\xi_1}^{2s}$ term by
integrating by parts when $\abs{\xi_3} \leq \abs{\xi_2}$. For the $\abs{\xi_2}^{2s}$ term, we add 
the following term to our modified energy:
$$\int_{\xi_2, \xi_3} \left(\frac12 A^2 \int_{\xi_1 \leq \xi_2, \xi_3}\abs{\xi_1}^2 
\abs{\hat u(\xi_1)}^2\right) E^s_{\xi_2 \xi_3}\vec{d\xi}. 
$$
Note that here $\xi_1$ is playing the role of the low frequency coefficient, which will 
slightly complicate things below.

To account for all these corrections simultaneously, we view the total corrections we need as 
forming a Taylor series in 
$$\int_{\abs{\xi_3} \leq \abs{\xi_1}, \abs{\xi_2}} \abs{\xi_3}^2 \abs{u(\xi_3)}^2 d\xi_3.$$
To calculate the function we need exactly, we assume that the corrections do indeed converge
to a function which we will denote 
$$F(\xi_1 \wedge \xi_2).$$ 
We further assume (and will verify 
in the next subsection) that when time derivatives land on $F$ we have control, which is 
again our theme because
$F$ depends on $u$ through only low frequencies.

We change our original correction to 
\begin{align*}
    \partial_t &\left(\int_{\xi_1} E^s_{\xi_1} + \int_{\xi_1,\xi_2} A 
    F(\xi_1 \wedge \xi_2)E^s_{\xi_1\xi_2}\vec{d\xi} \right)\\
&= \int_{\xi_1, \xi_2} A\abs{\xi_1}^{2+2s}\abs{\xi_2}^2 \abs{\hat u(\xi_1)}^2 
\Re(\hat u(\xi_2) \bar {\hat u}(\xi_2))\vec{d\xi} \\
&- \int_{\xi_1, \xi_2} AF(\xi_1 \wedge \xi_2)
\abs{\xi_1}^{2+2s}\abs{\xi_2}^2 \abs{\hat u(\xi_1)}^2 
\Re(\hat u(\xi_2) \bar {\hat u}(\xi_2))\vec{d\xi} \\
& - \int_{\xi_1,\xi_2,\xi_3} \frac12 A^2F(\xi_1 \wedge \xi_2)
(\abs{\xi_1}^{2s} - \abs{\xi_2}^{2s})\abs{\xi_1}^2\abs{\xi_2}^2\abs{\xi_3}^2
\abs{\hat u(\xi_1)}^2 \Re(\hat u(\xi_2) \bar {\hat u}'(\xi_2)) \abs{\hat u(\xi_3)}^2\vec{d\xi}\\
&+ (\text{terms where the derivative lands on $F$})
\end{align*}

For these to cancel (after controlling the favorable balances and integrating by parts)
we get an integral equality for $F$. What complicates things is we alter 
the input of $F$ when interchanging the symbols $\xi_1$ and $\xi_3$. Ultimately, the terms we 
cannot control combine into
\begin{align*}
    \int_{\xi_1, \xi_2} E^s_{\xi_1 \xi_2}\Bigg(
 &A - AF(\xi_1 \wedge \xi_2) -A^2F(\xi_1\wedge \xi_2)
\int_{\abs{\xi_3} \leq \xi_1 \wedge \xi_2} \abs{\xi_3}^2\abs{\hat u(\xi_3)}^2 d\xi_3\\
&- \frac12A^2 \int_{\abs{\xi_3} \leq \xi_1 \wedge \xi_2} F(\xi_3 \wedge \xi_2) 
\abs{\xi_3}^2\abs{\hat u(\xi_3)}^2d\xi_3
\Bigg) \vec{d\xi}.
\end{align*}
For this to cancel for all $\xi_1, \xi_2$ we need 
\begin{align*}
    0 = A &- AF(\xi_1 \wedge \xi_2) -A^2F(\xi_1\wedge \xi_2)
\int_{\abs{\xi_3} \leq \xi_1 \wedge \xi_2} \abs{\xi_3}^2\abs{\hat u(\xi_3)}^2 d\xi_3\\
          &- \frac12A^2 \int_{\abs{\xi_3} \leq \xi_1 \wedge \xi_2} F(\abs{\xi_3}) 
\abs{\xi_3}^2\abs{\hat u(\xi_3)}^2d\xi_3
\end{align*}
where we have replaced the minimum $\xi_3 \wedge \xi_2$ with $\abs{\xi_3}$ in the 
last term since $\abs{\xi_3} \leq \abs{\xi_2}$.

Since the size $\xi_1 \wedge \xi_2$ is a real number we can form an ODE for $F$ in 
terms of this 
parameter which we will call $r$. First, we see that $F(0) = 1$ because the 
integrals vanish when $r=0$. Then differentiating we get 
\begin{align*}
0= &-F'(r)\left(1+ A \int_{\abs{\xi_3} \leq r} \abs{\xi_3}^2\abs{\hat u(\xi_3)}^2d\xi_3\right)
- A F(r) \int_{\abs{\xi_3} = r} \abs{\xi_3}^2\abs{\hat u(\xi_3)}^2d\xi_3\\
   &- \frac12A F(r)\int_{\abs{\xi_3} = r} 
\abs{\xi_3}^2\abs{\hat u(\xi_3)}^2d\xi_3
\end{align*}
so that 
$$\frac{F'(r)}{F(r)} = \frac{- \frac32 A \int_{\abs{\xi_3} = s} \abs{\xi_3}^2
\abs{\hat u(\xi_3)}^2d\xi_3}{1 + A\int_{\abs{\xi_3} \leq r} \abs{\xi_3}^2 \abs{\hat u(\xi_3)}^2
d\xi_3}.$$
Noticing that the top is a multiple of the derivative of the bottom we see that 
$$F(r) = \left(1 + A\int_{\abs{\xi_3} \leq r} \abs{\xi_3}^2 \abs{\hat u(\xi_3)}^2d\xi_3\right)^{-3/2}.$$

\subsection{Analysis of the Correction} \label{analysis}
First we will need some basic facts about this higher order correction function $F$.
\begin{definition}\label{def:f}
Let $u \in C_t^0([0,T], H_x^{s_0+1}) \cap C_t^1([0,T], H_x^{s_0})$ solve \eqref{eq:kirchhoff} in the model 
case 
with $s_0 \geq 0$. We define $F(r)$ to be the unique solution to the integral equation 
$$F(r) = 1 - A F(r) \int_{\abs{\xi} \leq r} \abs{\xi}^2 \abs{\hat u(\xi)}^2 d\xi
- \frac12 A\int_{\abs{\xi} \leq r} F(\abs{\xi})\abs{\xi}^2 \abs{\hat u(\xi)}^2 d\xi,$$
which has the explicit form
$$F(r) = \left(1 + A\int_{\abs{\xi} \leq r} \abs{\xi}^2 \abs{\hat u(\xi)}^2d\xi\right)^{-3/2}.$$
\end{definition}
The essential analytic properties of $F$ we need are given by the following lemma.
\begin{lemma}
    \label{lem:f}
Let $u \in C_t^0([0,T], H_x^{s_0+1}) \cap C_t^1([0,T], H_x^{s_0})$ solve \eqref{eq:kirchhoff} in the model 
case  with $s_0 \geq 0$. Suppose that either $A \geq 0$ or for some time $t \in [0,T]$, we have
$$\norm{u(t)}_{\dot H^1} \leq \frac1{\sqrt {-2A}}.$$
Then, for $F$ as defined in Definition \ref{def:f} we have at time $t$ that 
$$F(r) \lesssim 1,$$
and
$$\abs{\partial_t F(r)} \lesssim \abs{A}\abs{\int_{\abs{\xi} \leq r} \abs{\xi}^2 \Re(\hat u(\xi)
\bar {\hat u}'(\xi))d\xi}$$
\end{lemma}
We leave the details to Subsection \ref{app:f} of the appendix.

Now we are ready to construct the modified energy for \eqref{eq:kirchhoff} in the model case.
\begin{definition}\label{def:model_energy}
Let $u \in C_t^0([0,T], H_x^{s_0+1}) \cap C_t^1([0,T], H_x^{s_0})$ solve \eqref{eq:kirchhoff} in the model case
with $s_0 \geq 0$. We define the \emph{modified energy} at regularity $s$ for \eqref{eq:kirchhoff}
in the model case to be 
$$E^s = \int_{\xi_1} E^s_{\xi_1}d\xi_1 + \int_{\xi_1,\xi_2} A 
    F(\xi_1 \wedge \xi_2)E^s_{\xi_1\xi_2}\vec{d\xi}  + E_n^s$$
where the energy densities $E^s_{\xi_1}$ and $E^s_{\xi_1\xi_2}$ are defined in Definitions \ref{def:unmodified} and \ref{def:coefficients} respectively; the higher order function $F$ is 
defined in Definition \ref{def:f}; and we collect the terms taking advantage of integrating by parts 
in 
\begin{align*}
    E^s_n &= -\int_{\abs{\xi_3} \leq \abs{\xi_1}, \abs{\xi_2}} \frac14 A^2 F(\xi_1 \wedge \xi_2)
\abs{\xi_2}^{2s}\abs{\xi_1}^2\abs{\xi_2}^2\abs{\xi_3}^2
\abs{\hat u(\xi_1)}^2 \abs{\hat u(\xi_2)}^2 \abs{\hat u(\xi_3)}^2\vec{d\xi} \\
          &- \int_{\abs{\xi_1} \leq \xi_2\wedge \xi_3} \frac14 A^2 F(\abs{\xi_1})
\abs{\xi_1}^2 \abs{\xi_2}^{2+2s}\abs{\xi_3}^2 \abs{\hat u(\xi_1)}^2\abs{\hat u(\xi_2)}^2
\abs{\hat u(\xi_3)}^2\vec{d\xi} \\
    &+ \int_{\abs{\xi_1} \leq \abs{\xi_3}\leq \abs{\xi_2}} \frac14 A^2 
    F(\abs{\xi_1})\abs{\xi_1}^{2s}\abs{\xi_1}^2\abs{\xi_2}^2\abs{\xi_3}^2
\abs{\hat u(\xi_1)}^2 \abs{\hat u(\xi_2)}^2 \abs{\hat u(\xi_3)}^2\vec{d\xi} .
\end{align*}
We use the notation $E^s_n$ to reflect the fact that these terms represent usual 
normal form integrals.
\end{definition}
Then, we have the following modified energy bound.
\begin{proposition} \label{prop:simple_energy}
Let $u \in C_t^0([0,T], H_x^{s_0+1}) \cap C_t^1([0,T], H_x^{s_0})$ solve \eqref{eq:kirchhoff} in 
the model case with 
$s_0 \geq 1/4$ and let 
$E^s$ be the modified energy associated to $u$ defined in Definition \ref{def:model_energy}. Finally, suppose $u$ satisfies the conditions of Lemma \ref{lem:f} at $t \in [0,T].$
Then for all $0 \leq s \leq s_0$ we have the energy estimate
\begin{align*}
\abs{\partial_t E^s}
    & \lesssim_{s_0} \abs{A}^2\norm{(u,u')}^2_{\dot H_x^{1+s} \times \dot H_x^s}
    \norm{(u,u')}^4_{\dot H_x^{5/4} \times \dot H_x^{1/4}}\\ 
    &+  \abs{A}^3\norm{(u,u')}^2_{\dot H_x^{1+s} \times \dot H_x^s}
    \norm{(u,u')}^4_{\dot H_x^{5/4} \times \dot H_x^{1/4}}\norm{(u,u')}^2_{\dot H_x^{1} \times L_x^2}.
\end{align*}
\end{proposition}
The proof of this proposition follows largely from the discussion in Section 
\ref{ssec:necessity}, Proposition \ref{prop:coefficients} and Lemma \ref{lem:f}.
\begin{proof}
We handle the derivatives one term at a time. We previously calculated that
$$\partial_t \int_{\xi_1} E^s_{\xi_1}d\xi_1 = \int_{\xi_1, \xi_2} A 
\abs{\xi_1}^{2+2s}\abs{\xi_2}^2 \abs{\hat u(\xi_1)}^2 
\Re(\hat u(\xi_2) \bar {\hat u}'(\xi_2))\vec{d\xi}.$$
Proposition \ref{prop:coefficients} as well as the identity for $F$ in Lemma \ref{lem:f} 
tell us that this is canceled by $\partial_t \int_{\xi_1,\xi_2} A 
F(\xi_1 \wedge \xi_2)E^s_{\xi_1\xi_2}\vec{d\xi} $ with the cost of the following remainders.
We get from Proposition \ref{prop:coefficients}:
\begin{equation}
\label{eq:error}
- \int_{\xi_1,\xi_2,\xi_3} \frac12 A^2F(\xi_1 \wedge \xi_2) 
(\abs{\xi_1}^{2s} - \abs{\xi_2}^{2s})\abs{\xi_1}^2\abs{\xi_2}^2\abs{\xi_3}^2
\abs{\hat u(\xi_1)}^2 \Re(\hat u(\xi_2) \bar {\hat u}'(\xi_2)) \abs{\hat u(\xi_3)}^2 \vec{d\xi};
\end{equation}
from the identity for $F$ in Lemma \ref{lem:f} (in the second term 
we have exchanged $\xi_1 \leftrightarrow \xi_3$ names to be consistent with how they will cancel):
\begin{align}
    \label{eq:1}&\int_{\abs{\xi_3} \leq \xi_1\wedge \xi_2} A^2 F(\xi_1 \wedge \xi_2) 
\abs{\xi_1}^{2+2s}\abs{\xi_2}^2\abs{\xi_3}^2 \abs{\hat u(\xi_1)}^2
    \Re(\hat u(\xi_2) \bar {\hat u}'(\xi_2)) \abs{\hat u(\xi_3)}^2\vec{d\xi}\\
    \label{eq:2}+ &\int_{\abs{\xi_1} \leq \xi_2\wedge \xi_3} \frac12 A^2 F(\abs{\xi_1})
\abs{\xi_1}^2 \abs{\xi_2}^2 \abs{\xi_3}^{2+2s}\abs{\hat u(\xi_1)}^2
    \Re(\hat u(\xi_2) \bar {\hat u}'(\xi_2))\abs{\hat u(\xi_3)}^2\vec{d\xi} ;
\end{align}
and from terms that come from when the time derivative lands on $F$ (which we will 
deal with at the end):
$$ \int_{\xi_1, \xi_2} (\partial_t F(\xi_1 \wedge \xi_2)) E_{\xi_1 \xi_2}\vec{d\xi}.$$

To control these remainders we need cancellation from the derivative of $E^s_n$. Each term 
in $E^s_n$ gives 3 parts: when the derivative falls on the high frequency $u$ factors (which 
we have combined by symmetry below):
\begin{align}
    \label{eq:3}&-\int_{\abs{\xi_3} \leq \xi_1 \wedge \xi_2} \frac12 A^2 F(\xi_1 \wedge \xi_2)
\left(\abs{\xi_1}^{2s} + \abs{\xi_2}^{2s}\right)\abs{\xi_1}^2\abs{\xi_2}^2\abs{\xi_3}^2
    \abs{\hat u(\xi_1)}^2 \Re(\hat u(\xi_2) \bar {\hat u}'(\xi_2)) \abs{\hat u(\xi_3)}^2\vec{d\xi}\\
    \label{eq:4}&- \int_{\abs{\xi_1} \leq \xi_2\wedge \xi_3} \frac12 A^2 F(\abs{\xi_1})
\abs{\xi_1}^2 \left(\abs{\xi_2}^{2s} + \abs{\xi_3}^{2s}\right)\abs{\xi_2}^{2}\abs{\xi_3}^2 
    \abs{\hat u(\xi_1)}^2\Re(\hat u(\xi_2) \bar {\hat u}'(\xi_2)) \abs{\hat u(\xi_3)}^2\vec{d\xi}\\
    \label{eq:5}&+ \int_{\abs{\xi_1} \leq \abs{\xi_3}\leq \abs{\xi_2}} \frac12 A^2 F(\abs{\xi_1})
\abs{\xi_1}^{2s}\abs{\xi_1}^2\abs{\xi_2}^2\abs{\xi_3}^2
    \abs{\hat u(\xi_1)}^2 \Re(\hat u(\xi_2) \bar {\hat u}'(\xi_2)) \abs{\hat u(\xi_3)}^2\vec{d\xi};
\end{align}
when the derivative lands on lower frequency $u$ factors:
\begin{align}
\label{eq:en1} &-\int_{\abs{\xi_3} \leq \xi_1 \wedge \xi_2} \frac12 A^2 F(\xi_1 \wedge \xi_2)
\abs{\xi_2}^{2s}\abs{\xi_1}^2\abs{\xi_2}^2\abs{\xi_3}^2
    \abs{\hat u(\xi_1)}^2 \abs{\hat u(\xi_2)}^2 \Re(\hat u(\xi_3) \bar {\hat u}'(\xi_3))\vec{d\xi} \\
 \label{eq:en2}          &- \int_{\abs{\xi_1} \leq \xi_2\wedge \xi_3} \frac12 A^2 F(\abs{\xi_1})
\abs{\xi_1}^2 \abs{\xi_2}^{2s} \abs{\xi_2}^{2}\abs{\xi_3}^2 
    \Re(\hat u(\xi_1) \bar {\hat u}'(\xi_1)) \abs{\hat u(\xi_2)}^2 \abs{\hat u(\xi_3)}^2\vec{d\xi} \\
\label{eq:en3} &+ \int_{\abs{\xi_1} \leq \abs{\xi_3}\leq \abs{\xi_2}} \frac12 A^2 F(\abs{\xi_1})
\abs{\xi_1}^{2s}\abs{\xi_1}^2\abs{\xi_2}^2\abs{\xi_3}^2
    \Re(\hat u(\xi_1) \bar {\hat u}'(\xi_1)) \abs{\hat u(\xi_2)}^2 \abs{\hat u(\xi_3)}^2\vec{d\xi} \\
\label{eq:en4} &+ \int_{\abs{\xi_1} \leq \abs{\xi_3}\leq \abs{\xi_2}} \frac12 A^2 F(\abs{\xi_1})
\abs{\xi_1}^{2s}\abs{\xi_1}^2\abs{\xi_2}^2\abs{\xi_3}^2
    \abs{\hat u(\xi_1)}^2\abs{\hat u(\xi_2)}^2 \Re(\hat u(\xi_3) \bar {\hat u}'(\xi_3))\vec{d\xi} ;
\end{align}
or when the derivative lands on $F$.

Again we will focus on the high frequency part and leave the other parts for later. We consider 
\eqref{eq:error} in cases depending on the smallest frequency:

\begin{itemize}
    \item($\abs{\xi_2} \leq \xi_1 \wedge \xi_2$): Exactly as in the discussion in 
        Section \ref{ssec:derivation} in this case \eqref{eq:error} 
        is bounded using Cauchy Schwartz and the fact that $\abs{F} \lesssim 1$
        $$\abs{\eqref{eq:error}} \lesssim \abs{A}^2\norm{(u,u')}^2_{\dot H_x^{1+s} \times \dot H_x^s}
    \norm{(u,u')}^4_{\dot H_x^{5/4} \times \dot H_x^{1/4}}$$

\item($\abs{\xi_3} \leq \xi_1 \wedge \xi_2$): Here we get exact 
    cancellation between \eqref{eq:error}, \eqref{eq:1}, and \eqref{eq:3}.

\item($\abs{\xi_1} \leq \xi_2 \wedge \xi_3$): Here we combine
    \eqref{eq:error}, \eqref{eq:2}, \eqref{eq:4}, and \eqref{eq:5} by noticing that in this case 
    $F(\xi_1 \wedge \xi_2) = F(\abs{\xi_1})$. Everything cancels except a term we can 
    control:
\begin{align*}
    - &\Big\lvert\int_{\abs{\xi_1} \leq \abs{\xi_2} \leq \abs{\xi_3}} \frac12 A^2F(\abs{\xi_1}) 
\abs{\xi_1}^{2s} \abs{\xi_1}^2\abs{\xi_2}^2\abs{\xi_3}^2
\abs{\hat u(\xi_1)}^2 \Re(\hat u(\xi_2) \bar {\hat u}'(\xi_2)) \abs{\hat u(\xi_3)}^2\vec{d\xi}
    \Big\rvert\\
      &\lesssim \abs{A}^2 \norm{(u,u')}^2_{\dot H_x^{1+s} \times \dot H_x^s}
    \norm{(u,u')}^4_{\dot H_x^{5/4} \times \dot H_x^{1/4}}
\end{align*}
since, in this region $\abs{\xi_2}^2\abs{\xi_3}^2 \leq \abs{\xi_2}^{3/2}\abs{\xi_3}^{5/2}$.

\end{itemize}
Next we consider terms where the time derivative falls on low frequency factors of $E_n^s$.
These will be controlled by 
    $$\abs{A}^2\norm{(u,u')}^2_{\dot H_x^{1+s} \times \dot H_x^s}
    \norm{(u,u')}^4_{\dot H_x^{5/4} \times \dot H_x^{1/4}}.$$ We can see this by using 
the following inequalities
on their integration regions:
\begin{align*}
    \abs{\xi_2}^{2s}\abs{\xi_1}^2\abs{\xi_2}^2\abs{\xi_3}^2 &\leq 
    \abs{\xi_1}^{5/2}\abs{\xi_2}^{2s+2}\abs{\xi_3}^{3/2} 
&& (\eqref{eq:en1}: {\abs{\xi_3} \leq \xi_1 \wedge \xi_2})\\
\abs{\xi_1}^2 \abs{\xi_2}^{2s} \abs{\xi_2}^{2}\abs{\xi_3}^2 &\leq 
\abs{\xi_1}^{3/2} \abs{\xi_2}^{2s+2} \abs{\xi_3}^{5/2}
&&(\eqref{eq:en2}: {\abs{\xi_1} \leq \xi_2\wedge \xi_3})\\
\abs{\xi_1}^{2s}\abs{\xi_1}^2\abs{\xi_2}^2\abs{\xi_3}^2 &\leq 
\abs{\xi_1}^{2s+1}\abs{\xi_2}^{5/2}\abs{\xi_3}^{5/2}
&&(\eqref{eq:en3}: {\abs{\xi_1} \leq \abs{\xi_3}\leq \abs{\xi_2}})\\
\abs{\xi_1}^{2s}\abs{\xi_1}^2\abs{\xi_2}^2\abs{\xi_3}^2 &\leq 
\abs{\xi_1}^{2s+2}\abs{\xi_2}^{5/2}\abs{\xi_3}^{3/2}
&&(\eqref{eq:en4}:{\abs{\xi_1} \leq \abs{\xi_3}\leq \abs{\xi_2}})
\end{align*}

Finally, when the derivative lands on $F$ we end up in a similar situation to when the 
derivative lands on low frequency in $u$ since $F$ depends on the lowest frequency when it 
appears. We compute examples from each term for the sake concision since all terms are 
controlled similarly:

From 
$$ \int_{\xi_1, \xi_2} (\partial_t F(\xi_1 \wedge \xi_2)) A E_{\xi_1 \xi_2}\vec{d\xi} $$
the result follows from the computations in Proposition \ref{prop:quad_bounds}, so we just compute 
the term coming from $c$ here.
We use Lemmas \ref{lem:kernel} and \ref{lem:f} as well as a relabeling in 
$\xi_1$ and $\xi_2$ to
make $\abs{\xi_1} \leq \abs{\xi_2}:$
\begin{align*}
    \Big\lvert\int_{\xi_1,\xi_2}& \partial_t F(\xi_1 \wedge \xi_2) \frac14 A \abs{\xi_1}^2 \abs{\xi_2}^2 
\frac{\abs{\xi_1}^{2s} - \abs{\xi_2}^{2s}}{\abs{\xi_1}^2 - \abs{\xi_2}^2} 
\Re(\hat u(\xi_1) \bar {\hat u}'(\xi_1) \hat u(\xi_2) \bar {\hat u}'(\xi_2))\vec{d\xi}\Big\rvert\\
& \lesssim \int_{\abs{\xi_3} \leq \abs{\xi_1} \leq \abs{\xi_2}}
\Big\lvert(1+s)A^2\abs{\xi_1}^2 \abs{\xi_2}^{2s} \abs{\xi_3}^2 
\Re(\hat u(\xi_1) \bar {\hat u}'(\xi_1)) \Re(\hat u(\xi_2) \bar {\hat u}'(\xi_2))
\Re(\hat u(\xi_3) \bar {\hat u}'(\xi_3))\Big\rvert \vec{d\xi} \\
& \leq (1+s) \abs{A}^2\norm{(u,u')}^2_{\dot H_x^{1+s} \times \dot H_x^s}
    \norm{(u,u')}^4_{\dot H_x^{5/4} \times \dot H_x^{1/4}}
\end{align*}
since on this region
$$\abs{\xi_1}^2 \abs{\xi_2}^{2s} \abs{\xi_3}^2
\leq\abs{\xi_1}^{3/2}\abs{\xi_2}^{2s+1}\abs{\xi_3}^{3/2}.$$
Note the reason this works is because the multiplier in Lemma \ref{lem:kernel} 
is always bounded by the higher frequency, also this is where the dependence on $s$ 
comes from.

Lastly when the derivative lands on $F$ in $E^s_n$ terms we end up with more copies of $u$ so we get a different bound. All work the 
same so let's look at the first term \eqref{eq:en1}. We see 
\begin{align*}
    \Big\lvert-&\int_{\abs{\xi_3}\leq \xi_1\wedge \xi_2} 
    \frac14 A^2 \partial_t F(\xi_1 \wedge \xi_2)
\abs{\xi_2}^{2s}\abs{\xi_1}^2\abs{\xi_2}^2\abs{\xi_3}^2
\abs{\hat u(\xi_1)}^2 \abs{\hat u(\xi_2)}^2 \abs{\hat u(\xi_3)}^2\vec{d\xi}\Big\rvert \\
&\lesssim\int_{\abs{\xi_3}, \abs{\xi_4} \leq \xi_1\wedge \xi_2}
    \Big\lvert A^3
    \abs{\xi_2}^{2s}\abs{\xi_1}^2\abs{\xi_2}^2\abs{\xi_3}^2\abs{\xi_4}^2
\abs{\hat u(\xi_1)}^2 \abs{\hat u(\xi_2)}^2 \abs{\hat u(\xi_3)}^2 
\Re(\hat u(\xi_4) \bar {\hat u}'(\xi_4))\Big\rvert \vec{d\xi} \\
& \leq  \abs{A}^3\norm{(u,u')}^2_{\dot H_x^{1+s} \times \dot H_x^s}
\norm{(u,u')}^4_{\dot H_x^{5/4} \times \dot H_x^{1/4}}\norm{(u,u')}^2_{\dot H_x^1 \times L_x^2}
\end{align*}
since on this region 
$$\abs{\xi_1}^2 \abs{\xi_2}^{2s+2} \abs{\xi_3}^2\abs{\xi_4}^2
\leq\abs{\xi_1}^{5/2}\abs{\xi_2}^{2s+2}\abs{\xi_3}^{2}\abs{\xi_4}^{3/2}.$$

\end{proof}
\subsection{Smallness of the Corrections}
In this section we will prove that the corrections we made to the energy are perturbative.
First, we note that we can estimate the corrections as follows:

\begin{lemma}
    \label{lem:corrections}
    Let $u \in C_t^0([0,T], H_x^{s_0+1}) \cap C_t^1([0,T], H_x^{s_0})$ solve \eqref{eq:kirchhoff} in the 
    model case with 
$s_0 \geq 0$. Further suppose the conditions of Lemma \ref{lem:f} are satisfied at $t \in [0,T]$. Then at time $t$ we have 
    $$\int_{\xi_1,\xi_2} \abs{ A 
        F(\xi_1 \wedge \xi_2)E^s_{\xi_1\xi_2}}\vec{d\xi}  \leq \frac32 (1+s)\abs{A} 
        \norm{(u,u')}^2_{\dot H_x^{1+s} \times \dot H_x^s} 
        \norm{(u,u')}^2_{\dot H_x^{1} \times L_x^2}$$
    and 
    $$\abs{E_n^s} \leq \frac32 \abs{A}^2 \norm{(u,u')}^2_{\dot H_x^{1+s} \times \dot H_x^s}
    \norm{(u,u')}^4_{\dot H_x^1 \times L_x^2}$$
\end{lemma}
\begin{proof}
    The first bound follows from Proposition \ref{prop:quad_bounds} and Lemma \ref{lem:f}.
    The second bound follows from the definition of $E^s_n$ in Definition 
    \ref{def:model_energy} and Lemma \ref{lem:f}.
\end{proof}

Next we will prove Theorem \ref{theorem:main} in the model case which we have restated below 
for convenience with the small constant replaced by its dependence on $A$:
\begin{manualtheorem}{1}[Model case]
Suppose that $u \in C_t^0([0,T], H_x^{s_0+1}) \cap C_t^1([0,T], H_x^{s_0})$, $s_0 \geq 1/4$ 
    solves \eqref{eq:kirchhoff} in the model case. Then there exists a family of modified energies 
    $E^s(t)$ with the property that for any $t \in [0, T]$ for which 
    $$ \norm{(u(t), u'(t))}_{\dot H_x^{1} \times L_x^2} \leq \delta = \frac1{\sqrt{8(1+s_0)A}}$$
    then for all $0 \leq s \leq s_0$
    \begin{equation}\label{eq:model_diff}
    \partial_t E^{s}(t) \lesssim_{s_0} E^{s}(t) (E^{1/4}(t))^2
    \end{equation}
    and  
    \begin{equation}\label{eq:model_sim}
    \norm{(u(t), u'(t))}^2_{\dot H_x^{1+s} \times \dot H_x^{s}} \sim_{s_0} E^{s}(t).
    \end{equation}
\end{manualtheorem}
\begin{proof}
Note that the smallness condition on the $\dot H^1_x \times L^2_x$ size of $(u(t), u'(t))$ 
implies the condition in Lemma \ref{lem:f}. In particular, we have the lower bound
$$ 1 -\frac14 \leq 1 + A\norm{u}_{\dot H^1_x}^2 = 1 + N$$

    First we use the fact that our corrections were small to show that 
    our modified energy controls the $\dot H_x^{s+1} \times \dot H_x^s$ norm.
    From Lemma \ref{lem:corrections}, the definition of $E^s$ in Definition \ref{def:model_energy}, 
    and the assumption that 
$$\norm{(u,u')}_{\dot H_x^1 \times L_x^2} \leq \delta$$
we get that 
    \begin{align*}
    \norm{(u(t), u'(t))&}^2_{\dot H_x^{1+s} \times \dot H_x^s} 
    \leq \norm{u(t)}_{\dot H_x^{1+s}}^2 + \norm{u'(t)}_{\dot H_x^s}^2 \\
    &\leq \frac43 (1+N)\norm{u(t)}_{\dot H_x^{1+s}}^2 + \norm{u'(t)}_{\dot H_x^s}^2 \\
    &\leq \frac83\int_{\xi_1} E^s_{\xi_1} d\xi_1\\
    &\leq \frac83\abs{E^s} +
    \frac83\abs{\int_{\xi_1,\xi_2} A 
    F(\xi_1 \wedge \xi_2)E^s_{\xi_1\xi_2}\vec{d\xi} } + \frac83\abs{E^s_n}\\
& \leq \frac83 \abs{E^s} + 4 \abs{A}(1+s) \delta^2 
\norm{(u(t), u'(t))}^2_{\dot H_x^{1+s} \times \dot H_x^s} 
    + 4\abs{A}^2 \delta^4\norm{(u(t), u'(t))}^2_{\dot H_x^{1+s} \times \dot H_x^s}\\
    & \leq \frac83 \abs{E^s} + \frac12 \norm{(u(t), u'(t))}^2_{\dot H_x^{1+s} \times \dot H_x^s} 
+ \frac1{16}\norm{(u(t), u'(t))}^2_{\dot H_x^{1+s} \times \dot H_x^s}.\\
    \end{align*}
In particular, we may absorb the terms on the right into the left hand side to arrive at 
$$\norm{(u(t), u'(t))}^2_{\dot H_x^{1+s} \times \dot H_x^s} \lesssim \abs{E^s}.$$
This proves one direction of \eqref{eq:model_sim}.

Note that indeed $E^s(t) \geq 0$ which follows from the fact that our 
corrections are smaller than the unmodified energy:
\begin{align*}
    \Big\vert\int_{\xi_1,\xi_2} A 
    &F(\xi_1 \wedge \xi_2)E^s_{\xi_1\xi_2}\vec{d\xi}\Big\vert  + \abs{E^s_n}\\
    &\leq  \frac34 A (1+s)\delta^2 \norm{(u(t), u'(t))}^2_{\dot H_x^{1+s} \times \dot H_x^s} 
    + \frac34 A^2 \delta^4\norm{(u(t), u'(t))}^2_{\dot H_x^{1+s} \times \dot H_x^s}\\
                                                         &\leq \frac7{64} 
    \norm{(u(t), u'(t))}^2_{\dot H_x^{1+s} \times \dot H_x^s}\\
    &\leq \frac14 (1+N)\norm{u(t)}_{\dot H_x^{1+s}}^2 + \frac14 \norm{u'(t)}_{\dot H_x^s}^2\\
    &= \frac12 \int_{\xi_1} E^s_{\xi_1} d\xi_1
\end{align*}

To get the other direction of \eqref{eq:model_sim}, we use the fact 
that the nonlinearity $N$ is bounded when $\norm{(u(t),u'(t))}_{\dot H_x^1 \times L_x^2}$ is bounded
as well as the fact that our energy corrections are smaller than the unmodified energy.
In particular 
\begin{align*}
    \abs{E^s} &\leq 
    \abs{\int E^s_{\xi_1} d\xi_1} + \abs{\int_{\xi_1,\xi_2} A 
    F(\xi_1 \wedge \xi_2)E^s_{\xi_1\xi_2}\vec{d\xi} } + \abs{E^s_c}\\ 
              &\lesssim \int E^s_{\xi_1} d\xi_1\\
    &= \frac12 (1+N)\norm{u(t)}_{\dot H_x^{1+s}}^2 + \frac12 \norm{u'(t)}_{\dot H_x^s}^2\\
    &\lesssim_N \norm{u(t)}_{\dot H_x^{1+s}}^2 + \norm{u'(t)}_{\dot H_x^s}^2
\end{align*}

Finally the differential inequality \eqref{eq:model_diff} follows from Proposition \ref{prop:simple_energy} and 
the fact that 
$$\norm{(u,u')}^2_{\dot H_x^1 \times L_x^2} \leq \delta^2 \lesssim_{A,s_0} 1.$$

In particular, Proposition \ref{prop:simple_energy} gives us 
\begin{align*}
    \abs{\partial_t E^{s}(t)} 
    &\lesssim_{s_0} A^2\norm{(u,u')}^2_{\dot H_x^{1+s} \times \dot H_x^s}
    \norm{(u,u')}^4_{\dot H_x^{5/4} \times \dot H_x^{1/4}}\\
    &\qquad + 
    A^3\norm{(u,u')}^2_{\dot H_x^{1+s} \times \dot H_x^s}
    \norm{(u,u')}^4_{\dot H_x^{5/4} \times \dot H_x^{1/4}}\norm{(u,u')}^2_{\dot H_x^{1} \times L_x^2}\\
    &\lesssim_{s_0,A} E^s(t)(E^{1/4}(t))^2
\end{align*}

\end{proof}


\section{The General Case} \label{gen}

Before we begin, let us note that similarly to the model case in Section \ref{sec:simple}, 
we will require 
$$\norm{(u(t), u(t))}_{\dot H^1_x \times L^2} \leq \delta$$
for some small constant $\delta$ to be chosen below.
Thus, inputs to $N$ will live in the compact interval $[0,\delta^2]$ which is why 
we only need control of $N$ and its derivatives on a compact interval near 0.

To generalize our nonlinearity to general $C^2_{loc}$ functions we would like to replace $A$ 
in the model case above with $N'$. This, however, fails because $N'$ depends on time 
through high frequency components of $u$ which we will not be able to control when derivatives 
land on $N'$. Instead we use a decomposition of a more paradifferential flavor.

First, suppose that $N$ is analytic. Since $N(0) = 0$ we can expand and consider the 
factor bringing the highest frequency:
\begin{align*}
    N\Bigg(\int_{\xi} &\abs{\xi}^2 \abs{\hat u(\xi)}^2 d\xi\Bigg) 
= \sum_{i \geq 1} c_i \left(\int_{\xi} \abs{\xi}^2 \abs{\hat u(\xi)}^2 d\xi\right)^i\\
&= \sum_{i \geq 1} c_i \int_{\xi_1, \ldots, \xi_i} \abs{\xi_1}^2 \abs{\hat u(\xi_1)}^2 \cdots 
\abs{\xi_i}^2 \abs{\hat u(\xi_i)}^2 \vec{d\xi}\\
&= \sum_{i \geq 1} c_i \left(\int_{\abs{\xi_1} \geq \abs{\xi_2}, \ldots, \abs{\xi_i}} 
    \abs{\xi_1}^2 \abs{\hat u(\xi_1)}^2 \cdots d\xi
    + \cdots +\int_{\abs{\xi_i} \geq \abs{\xi_1}, \ldots, \abs{\xi_{i-1} }} \abs{\xi_i}^2 
\abs{\hat u(\xi_i)}^2 \cdots d\vec{\xi} \right)\\
&= \sum_{i \geq 1} c_i i \int_{\xi} \abs{\xi}^2\abs{\hat u(\xi)}^2 \left(\int_{\abs{\xi} \geq \abs{\eta}}
\abs{\eta}^2 \abs{\hat u(\eta)}^2 d\eta\right)^{i-1}d\xi\\
&= \int_{\xi} \abs{\xi}^2\abs{\hat u(\xi)}^2 N'\left(\int_{\abs{\eta} \leq \abs{\xi}} 
\abs{\eta}^2 \abs{\hat u(\eta)}^2 d\eta\right) d\xi
\end{align*}

Note that this is akin to the usual paradifferential expansion where we do not distinguish the 
high high interactions from the high low interactions. We will see that it is enough to 
deal with them in the same way.

Further, notice that this is simply a paradifferential interpretation of the fundamental theorem of calculus. If we put 
$$f(r) = N\Bigg(\int_{\abs{\xi} \leq r} \abs{\xi}^2 \abs{\hat u(\xi)}^2 d\xi\Bigg),$$
then the above equality can be deduced from 
$$f(\infty) = \int_0^\infty f'(r) dr.$$

Now, the $N'$ term is playing the role of $A$ in the model case because it depends on 
frequencies less than $\xi$, so when the derivative falls here we will still have control. This 
is a useful enough comparison that we make the following definition.
\begin{definition}
Let $u \in C_t^0([0,T], H_x^{s_0+1}) \cap C_t^1([0,T], H_x^{s_0})$ solve \eqref{eq:kirchhoff} in the 
general case with $s_0 \geq 0$. We define 
$$A(r) = N'\left(\int_{\abs{\eta} \leq r}
\abs{\eta}^2 \abs{\hat u(\eta)}^2 d\eta\right)$$
which parallels the constant $A$ in Section \ref{sec:simple}.
\end{definition}
Now since we control its derivatives, we can think of $A$ as being just as constant as in 
Section \ref{sec:simple}.
We can see explicitly how we get this control in the derivative of the unmodified linear energy 
\eqref{eq:quad_energy}.
We get 
cancellation when the derivative falls on the linear part from the equation; we get what 
we had in the model case when the derivative misses $N'$; and the new term involves $N''$:
\begin{align}
    \partial_t &\int E^s_{\xi_1} d\xi_1 
= \partial_t \int_{\xi_1} \frac12\left(1 + \int_{\xi_2}\abs{\xi_2}^2
\abs{\hat u(\xi_2)}^2d\xi_2\right) \abs{\xi_1}^{2s+2} \abs{\hat u(\xi_1)}^2 
+ \frac12 \abs{\xi_1}^{2s}\abs{\hat u'(\xi_1)}^2d\xi_1\nonumber\\
    =& \label{eq:unsymmetrized} \int_{\xi_1, \xi_2}  
\abs{\xi_1}^{2+2s}A(\abs{\xi_2})\abs{\xi_2}^2 \abs{\hat u(\xi_1)}^2 
\Re(\hat u(\xi_2) \bar {\hat u}'(\xi_2))  
d\xi_1d\xi_2\\
    +& \int_{\abs{\xi_3} \leq \abs{\xi_2}}  
    \abs{\xi_1}^{2+2s}\abs{\xi_2}^2 \abs{\xi_3}^2\abs{\hat u(\xi_1)}^2  \abs{\hat u(\xi_2)}^2 
\Re(\hat u(\xi_3) \bar {\hat u}'(\xi_3)) 
N''\left(\int_{\abs{\xi_4} \leq \abs{\xi_3}} \abs{\xi_4}^2 \abs{\hat u(\xi_4)}^2 d\xi_4\right)
\vec{d\xi}\nonumber.
\end{align}
But it is clear that we can control the $N''$ term as long as $N''$ is uniformly bounded.

The only remaining obstruction to considering the function $A$ as a constant 
as in the model case is that we need the argument of $A$ to depend symmetrically on frequencies 
$\xi_1$ and $\xi_2$ as required in Proposition \ref{prop:coefficients}. To symmetrize 
$A$ in the correction we want $A$ to depend on the minimum frequency:
$$\int_{\xi_1,\xi_2} F(\xi_1 \wedge \xi_2) A(\xi_1 \wedge \xi_2)  E^s_{\xi_1 \xi_2} \vec{d\xi}.$$ 
However, we see that in \eqref{eq:unsymmetrized} $A$ depends only on $\abs{\xi_2}$.
Thus we would like to control the difference between the symmetric and antisymmetric versions: 
\begin{equation}\label{eq:difference}
-\int_{\xi_1,\xi_2} F(\xi_1 \wedge \xi_2) \Big( A(\abs{\xi_2}) - A(\xi_1 \wedge \xi_2)\Big)  E^s_{\xi_1 \xi_2} \vec{d\xi}. 
\end{equation}
We see by Taylor's theorem
\begin{align}
\abs{A(\abs{\xi_2}) - A(\xi_1 \wedge \xi_2)} &= \abs{N'\left(\int_{\abs{\xi_3} \leq \abs{\xi_2}} \abs{\xi_3}^2 \abs{\hat u(\xi_3)}^2 d\xi_3
    \right)
- N'\left(\int_{\abs{\xi_3} \leq \xi_1 \wedge \xi_2} \abs{\xi_3}^2 
\abs{\hat u(\xi_3)}^2 d\xi_3\right)}\nonumber \\& \label{eq:Taylor}\leq \norm{N''}_{L^\infty} \int_{\abs{\xi_1}
\leq 
\abs{\xi_3} \leq \abs{\xi_2}} \abs{\xi_3}^2 \abs{\hat u(\xi_3)}^2 d\xi_3.
\end{align}
Then \eqref{eq:difference} is controlled by 
$$ \norm{N''}_{L^\infty} \int_{\abs{\xi_1} \leq 
\abs{\xi_3} \leq \abs{\xi_2}} 
\abs{\xi_1}^{2+2s}\abs{\xi_2}^2 \abs{\xi_3}^2  \abs{\hat u(\xi_1)}^2
\Re(\hat u(\xi_2) \bar {\hat u}'(\xi_2)) \abs{\hat u(\xi_3)}^2\vec{d\xi} 
$$
Luckily, we can get control of this by integrating by parts as in Section \ref{ssec:derivation}.
Thus we need to add the following term to our modified energy
\begin{align*}
E^s_A=  -\int_{\abs{\xi_1} \leq \abs{\xi_2}} \frac12 \abs{\xi_1}^{2s+2} \abs{\xi_2}^2 
    \Big(A(\abs{\xi_2}) - A(\xi_1 \wedge \xi_2)\Big)
\abs{\hat u(\xi_1)}^2 \abs{\hat u(\xi_2)}^2\vec{d\xi} 
\end{align*}
which has the same bound as our second order corrections in Proposition \ref{prop:quad_bounds}.
Note that in the model case, $N'$ was a constant which is why this term did not appear. We denote 
this term $E_A^s$ to reflect that it arises from $A$.

The derivative of $E_A^s$ then produces \eqref{eq:difference}, which symmetrizes \eqref{eq:unsymmetrized} as well as remainders when the derivative lands on $\hat u(\xi_1)$ and 
on $A$. These are controlled by  
\begin{align*}
    \norm{N''}_{L^\infty} &\int_{\abs{\xi_1} \leq 
\abs{\xi_3} \leq \abs{\xi_2}} 
\abs{\xi_1}^{2+2s}\abs{\xi_2}^2 \abs{\xi_3}^2  
\Re(\hat u(\xi_1) \bar {\hat u}'(\xi_1))\abs{\hat u(\xi_2)}^2 \abs{\hat u(\xi_3)}^2\vec{d\xi} \\
        +&\norm{N''}_{L^\infty} \int_{\abs{\xi_1} \leq 
\abs{\xi_3} \leq \abs{\xi_2}} 
\abs{\xi_1}^{2+2s}\abs{\xi_2}^2 \abs{\xi_3}^2  
\abs{\hat u(\xi_1)}^2 \abs{\hat u(\xi_2)}^2\Re(\hat u(\xi_3) \bar {\hat u}'(\xi_3))\vec{d\xi}\\  
            +&\norm{N''}_{L^\infty} \int_{\abs{\xi_3} \leq \xi_1 \wedge \xi_2}
\abs{\xi_1}^{2+2s}\abs{\xi_2}^2 \abs{\xi_3}^2  
\abs{\hat u(\xi_1)}^2 \abs{\hat u(\xi_2)}^2\Re(\hat u(\xi_3) \bar {\hat u}'(\xi_3)) 
\vec{d\xi}.
\end{align*}
Each of these is controlled by $\norm{(u,u')}^2_{\dot H_x^{s+1} \times \dot H_x^s} 
\norm{(u,u')}^4_{\dot H_x^{5/4} \times \dot H_x^{1/4}}$.

With this modification the proof of Theorem \ref{theorem:main} follows in the same way as in 
the model case but with the constant also depending on the size of $N'$ and $N''$ on a small  
compact interval.

Because of the similarity we will state the modified energy needed and discuss briefly the 
modifications needed. We start
with the definitions which parallel the model case. The unmodified and leading order energy densities 
are identical and given in Definitions \ref{def:unmodified} and \ref{def:coefficients} respectively.

\begin{definition}\label{def:gen_f}
Let $u \in C_t^0([0,T], H_x^{s_0+1}) \cap C_t^1([0,T], H_x^{s_0})$ solve \eqref{eq:kirchhoff} 
with $s_0 \geq 0$. We define $F(r)$ to be the unique solution to the integral equation 
$$F(r) = 1 - F(r) \int_{\abs{\xi} \leq r} A(\abs{\xi}) \abs{\xi}^2 \abs{\hat u(\xi)}^2 d\xi
- \frac12 \int_{\abs{\xi} \leq r} F(\abs{\xi})A(\abs{\xi})\abs{\xi}^2 \abs{\hat u(\xi)}^2 d\xi,$$
which has the explicit form
$$F(r) = \left(1 + \int_{\abs{\xi} \leq r} A(\abs{\xi})\abs{\xi}^2 \abs{\hat 
u(\xi)}^2d\xi\right)^{-3/2} = \left(1 + N\left(\int_{\abs{\xi} \leq r} \abs{\xi}^2 \abs{\hat 
u(\xi)}^2d\xi\right)\right)^{-3/2} .$$   
\end{definition}
The two expressions for $F$ are equal because of the following identity
$$A(r) r^2 \int_{\abs{\xi} = r} \abs{\hat u(\xi)}^2 d\xi = \frac d{dr} 
N\left(\int_{\abs{\xi} \leq r} \abs{\xi}^2 \abs{\hat u(\xi)}^2 d\xi \right)$$
and the fact that $N(0) = 0.$ Further note that the continuity of $N$ means that 
by making $\norm{u(t)}_{\dot H^1}$ small enough we get that for all $r \geq 0$
$$1/2 \leq 1 + N\left(\int_{\abs{\xi} \leq r} \abs{\xi}^2 \abs{\hat 
u(\xi)}^2d\xi\right).$$
In particular, we have similar bounds as Lemma \ref{lem:f}.

Now we are in a position to define the modified energy needed.
\begin{definition}\label{def:full_energy}
Let $u \in C_t^0([0,T], H_x^{s_0+1}) \cap C_t^1([0,T], H_x^{s_0})$ solve \eqref{eq:kirchhoff} 
in the general case with $s_0 \geq 0$. We define the \emph{modified energy} at regularity $s$ for \eqref{eq:kirchhoff}
in the general case to be 
$$E^s = \int_{\xi_1} E^s_{\xi_1}d\xi_1 + \int_{\xi_1,\xi_2} A(\xi_1 \wedge \xi_2)
    F(\xi_1 \wedge \xi_2)E^s_{\xi_1\xi_2}\vec{d\xi}  + E_n^s + E^s_A$$
where the energy densities $E^s_{\xi_1}$ and $E^s_{\xi_1\xi_2}$ are defined in Definitions \ref{def:unmodified} and \ref{def:coefficients} respectively; the higher order function $F$ is 
defined in Definition \ref{def:gen_f}; we collect the terms taking advantage of integrating by parts 
in 
\begin{align*}
    E^s_n &= -\int_{\abs{\xi_3} \leq \abs{\xi_1}, \abs{\xi_2}} \frac14 
    A(\abs{\xi_3}) A(\xi_1 \wedge \xi_2) F(\xi_1 \wedge \xi_2)
    \abs{\xi_1}^2\abs{\xi_2}^{2+2s}\abs{\xi_3}^2
\abs{\hat u(\xi_1)}^2 \abs{\hat u(\xi_2)}^2 \abs{\hat u(\xi_3)}^2\vec{d\xi} \\
          &- \int_{\abs{\xi_1} \leq \xi_2\wedge \xi_3} \frac14 
          A(\abs{\xi_3}) A(\xi_1 \wedge \xi_2) F(\abs{\xi_1})
\abs{\xi_1}^2 \abs{\xi_2}^{2+2s}\abs{\xi_3}^2 \abs{\hat u(\xi_1)}^2\abs{\hat u(\xi_2)}^2
\abs{\hat u(\xi_3)}^2\vec{d\xi} \\
    &+ \int_{\abs{\xi_1} \leq \abs{\xi_3}\leq \abs{\xi_2}} \frac14  
    A(\abs{\xi_3}) A(\xi_1 \wedge \xi_2)
    F(\abs{\xi_1})\abs{\xi_1}^{2+2s}\abs{\xi_2}^2\abs{\xi_3}^2
\abs{\hat u(\xi_1)}^2 \abs{\hat u(\xi_2)}^2 \abs{\hat u(\xi_3)}^2\vec{d\xi} ;
\end{align*}
and we collect the term we need to control the discrepancy between symmetric and asymmetric $A$ 
terms in 
\begin{align*}
    E^s_A =  -\int_{\abs{\xi_1} \leq \abs{\xi_2}} \frac12 \abs{\xi_1}^{2s+2} \abs{\xi_2}^2
    \Big(A(\abs{\xi_2}) - A(\xi_1 \wedge \xi_2) \Big)
\abs{\hat u(\xi_1)}^2 \abs{\hat u(\xi_2)}^2\vec{d\xi}.
\end{align*}
We use the notation $E^s_n$ to reflect the fact that these terms represent gains from usual 
normal form integrals and $E^s_A$ to reflect the fact that this term depends on the non-constancy 
of $A$.
\end{definition}

\begin{manualtheorem}{1}[General case]
    \label{theorem:gen_main}
Suppose that $u \in C_t^0([0,T], H_x^{s_0+1}) \cap C_t^1([0,T], H_x^{s_0})$, $s_0 \geq 1/4$ 
    solves \eqref{eq:kirchhoff}. Then there exists a constant $\delta$ depending on $N$ and 
    $s_0$ and a family of modified energies 
    $E^s(t)$ with the property that for any $t \in [0, T]$ for which 
    $$ \norm{(u(t), u'(t))}_{\dot H_x^{1} \times L_x^2} \leq \delta$$
    then for all $0 \leq s \leq s_0$
    $$\partial_t E^{s}(t) \lesssim_{s_0} E^{s}(t) (E^{1/4}(t))^2$$
    and  
    $$\norm{(u(t), u'(t))}^2_{\dot H_x^{1+s} \times \dot H_x^{s}} \sim_{s_0} E^{s}(t).$$
\end{manualtheorem}
\begin{proof}
    Since $A$ above depends on low frequency, when the time derivative lands on it we will have 
    control. Thus, the only difference between this case and the model case is that the factors 
    of $A$ depend on different frequencies depending on where each arose from. The factors we 
    introduced needed symmetry and so they depend on $\xi_1 \wedge \xi_2$, but the factors 
    that arose from the equation depend on the introduced frequency $\xi_3$. Let's consider 
    the leading order part which corresponds to the labeled terms in the proof of Proposition
    \ref{prop:simple_energy}. 

We get the term coming from $\partial_t E^s_{\xi_1}$ after canceling with $E_A^s$
\begin{align}
    \label{eq:gen_quad} \int_{\xi_1,\xi_2} & A(\xi_1 \wedge \xi_2) 
\abs{\xi_1}^{2+2s}\abs{\xi_2}^2
\abs{\hat u(\xi_1)}^2 \Re(\hat u(\xi_2) \bar {\hat u}'(\xi_2))\vec{d\xi};
\end{align}
and terms coming from $\partial_t E^s_{\xi_1\xi_2}$
\begin{align}
    \label{eq:gen_quad_correct}- \int_{\xi_1,\xi_2} & A(\xi_1 \wedge \xi_2) 
    F(\xi_1 \wedge \xi_2) 
\abs{\xi_1}^{2s} \abs{\xi_1}^{2+2s}\abs{\xi_2}^2
\abs{\hat u(\xi_1)}^2 \Re(\hat u(\xi_2) \bar {\hat u}'(\xi_2)) \vec{d\xi}.\\
\label{eq:nerror}- \int_{\xi_1,\xi_2,\xi_3} &\frac12 A(\xi_1 \wedge \xi_2) A(\abs{\xi_3})
    F(\xi_1 \wedge \xi_2) 
(\abs{\xi_1}^{2s} - \abs{\xi_2}^{2s})\abs{\xi_1}^2\abs{\xi_2}^2\abs{\xi_3}^2\\
&\nonumber \abs{\hat u(\xi_1)}^2 \Re(\hat u(\xi_2) \bar {\hat u}'(\xi_2)) \abs{\hat u(\xi_3)}^2\vec{d\xi}.
\end{align}
From the integral identity for $F$ in \ref{def:gen_f} 
the terms \eqref{eq:gen_quad} and \eqref{eq:gen_quad_correct} cancel leaving the remainder below 
where in the second term we have switched the names
$\xi_1$ and $\xi_3$ 
because this is how they will cancel with \eqref{eq:nerror}.
\begin{align}
\label{eq:n1}
&\int_{\abs{\xi_3} \leq \xi_1\wedge \xi_2} A(\xi_1 \wedge \xi_2) A(\abs{\xi_3})
F(\xi_1 \wedge \xi_2) 
\abs{\xi_1}^{2+2s}\abs{\xi_2}^2\abs{\xi_3}^2 \abs{\hat u(\xi_1)}^2
    \Re(\hat u(\xi_2) \bar {\hat u}'(\xi_2)) \abs{\hat u(\xi_3)}^2\vec{d\xi}\\
\label{eq:n2}
    + &\int_{\abs{\xi_1} \leq \xi_2\wedge \xi_3} \frac12  A(\xi_3 \wedge \xi_2) A(\abs{\xi_1})
    F(\abs{\xi_1})
\abs{\xi_1}^2 \abs{\xi_2}^2 \abs{\xi_3}^{2+2s}\abs{\hat u(\xi_1)}^2
    \Re(\hat u(\xi_2) \bar {\hat u}'(\xi_2))\abs{\hat u(\xi_3)}^2\vec{d\xi}.
\end{align}
Now, just as in the proof of Proposition \ref{prop:simple_energy}, 
we control \eqref{eq:nerror} in all frequency 
balances except when 
$$\abs{\xi_3} \leq \xi_1 \wedge \xi_2, \quad \text{ or } \quad \abs{\xi_1} \leq \xi_2 \wedge \xi_3.$$
In the case when $\abs{\xi_3} \leq \xi_1 \wedge \xi_2$, the first term of $E^s_n$ helps us 
cancel \eqref{eq:n2} with \eqref{eq:nerror} along with terms which can be controlled.

In the other case when $\abs{\xi_1} \leq \xi_2 \wedge \xi_3$, we notice that by simplifying the
minimums and using the second and third terms of $E^s_n$, \eqref{eq:nerror} becomes 
\begin{align*}
\label{eq:nerror}- \int_{\abs{\xi_1}\leq \xi_2\wedge\xi_3} &\frac12 A(\abs{\xi_1}) A(\abs{\xi_3})
    F(\abs{\xi_1}) 
\abs{\xi_3}^{2s}\abs{\xi_1}^2\abs{\xi_2}^2\abs{\xi_3}^2\\
&\nonumber \abs{\hat u(\xi_1)}^2 \Re(\hat u(\xi_2) \bar {\hat u}'(\xi_2)) \abs{\hat u(\xi_3)}^2\vec{d\xi}.
\end{align*}
In the model case in Proposition \ref{prop:simple_energy}, this partially canceled form of 
\eqref{eq:nerror} and \eqref{eq:n2} canceled exactly since 
$A(\abs{\xi_3})$ and $A(\xi_3 \wedge \xi_2)$ were identical. Here, however, we use the Taylor 
expansion of $N'$ in \eqref{eq:Taylor} discussed above to control this difference. Again we 
get complete cancellation when $\abs{\xi_2} \leq \abs{\xi_3}$ so we are left with
\begin{align*}
\norm{N''}_{L^\infty} &\int_{\abs{\xi_1} \leq \abs{\xi_2} \leq \abs{\xi_4} \leq \abs{\xi_3}} 
\Big\lvert\frac12 A(\abs{\xi_1}) F(\abs{\xi_1})
    \abs{\xi_1}^2 \abs{\xi_2}^2 \abs{\xi_3}^{2+2s}\abs{\xi_4}^2\\ 
                      &\qquad \abs{\hat u(\xi_1)}^2
    \Re(\hat u(\xi_2) \bar {\hat u}'(\xi_2))\abs{\hat u(\xi_3)}^2\abs{\hat u(\xi_4)}^2
    \Big\rvert\vec{d\xi}\\
                      &\lesssim_N \norm{(u,u')}^2_{\dot H_x^{s+1} \times \dot H_x^s}
                      \norm{(u,u')}^4_{\dot H_x^{5/4} \times \dot H_x^{1/4}} 
                      \norm{(u,u')}^2_{\dot H_x^1 \times L_x^2}
\end{align*}
Lastly, we note that implicit in this discussion we have assumed that $F$ has similar bounds 
as in Lemma \ref{lem:f}. These bounds are proved in exactly the same way as in Lemma \ref{lem:f} 
with the only modification being that the bounds depend on the size of $A$ and its derivative 
which depend on the $C^2_{loc}$ size of $N$.
\end{proof}

\section{Implications to Well-Posedness}
\label{sec:implications}

In this section, we focus on the proofs of Theorems \ref{theorem:enhanced} and \ref{thm:weak_solutions} which are applications 
of the modified energy estimate in Theorem \ref{theorem:main}.

\subsection{Enhanced Lifespan}

We start with Theorem \ref{theorem:enhanced}, which proves quintic lifespan for solutions to \eqref{eq:kirchhoff} with initial 
data in Sobolev spaces where local well-posedness is known, with the important improvement that this lifespan only 
depends on the lower regularity Sobolev norm $\dot H^{5/4}_x \times \dot H^{1/4}_x$. We repeat the statement here 
for clarity.
\begin{manualtheorem}{\ref{theorem:enhanced}}
    Let $\delta$ be the small constant in Theorem \ref{theorem:main}.
    Suppose that the initial data $(u_0, u_1) \in H_x^{s_0+1} \times H_x^{s_0},$ $s_0 \geq 1/2$ has 
    $$\norm{(u_0, u_1)}_{\dot H_x^{5/4} \times \dot H_x^{1/4}} = \epsilon$$ 
    and
    $$\norm{(u_0, u_1)}_{\dot H_x^{1} \times L_x^2} = \delta_0 \ll_{s_0, \epsilon} \delta.$$
    Then 
    there exists a unique solution 
    $$u \in C_t^0([0,T], H_x^{s+1}) \cap C_t^1([0,T],H_x^{s})$$ 
    to \eqref{eq:kirchhoff} where $T \sim_{s_0} 1/\epsilon^4.$
    Furthermore we have the energy estimates on $[0,T]$ for all $0 \leq s \leq s_0$
    $$\sup_{0 \leq t \leq T}\norm{(u, \partial_t u)}_{\dot H_x^{s+1} \times \dot H_x^{s}} 
    \lesssim_{s_0, \epsilon} \norm{(u_0, u_1)}_{\dot H_x^{s+1} \times \dot H_x^s}.$$
    In particular
    $$\sup_{0 \leq t \leq T}\norm{(u, \partial_t u)}_{\dot H_x^{1} \times L_x^2} \leq \delta$$
    and
    $$\sup_{0 \leq t \leq T}\norm{(u, \partial_t u)}_{\dot H_x^{5/4} \times \dot H_x^{1/4}} \lesssim \epsilon.$$
\end{manualtheorem}
\begin{proof}
This follows from the Hadamard local well-posedness of \eqref{eq:kirchhoff} in 
$H_x^{s_0+1} \times H_x^{s_0}$ for $s_0 \geq 1/2$ and Theorem \ref{theorem:main}.

By the local well-posedness result (see \cite{arosio1996well}) we get a unique solution 
$u \in C_t^0([0,T], H_x^{s_0+1}) \cap C_t^1([0,T], H_x^{s_0})$. We want to show that 
as long as $T \lesssim_{s_0} 1/\epsilon^4$ the norm 
$\norm{(u(t), u'(t))}_{H_x^{s_0+1} \times H_x^{s_0}}$
has not grown by more than a uniform constant. This 
way, we can contradict the maximality of any final time of existence $T \ll 1/\epsilon^4$, 
by again applying 
the local well-posedness theory. We use a continuity argument.

First, let us make the constants in Theorem \ref{theorem:main} explicit. When
the conditions for Theorem \ref{theorem:main} are met, we will get a family of $E^s$ with
the following explicit constant bounds:
$$\abs{\partial_t E^s(t)} \leq C_{s_0} E^s(t) (E^{1/4}(t))^2$$
and 
$$\frac1{\tilde C_{s_0}} \norm{(u(t), u'(t))}^2_{\dot H_x^{s+1} \times \dot H_x^s} \leq 
E^s(t) \leq \tilde C_{s_0} \norm{(u(t), u'(t))}^2_{\dot H_x^{s+1} \times \dot H_x^s}.$$
Now, we can make precise the sizes of $\epsilon,$ $\delta_0$, and $T$. We put
$$\norm{(u_0, u_1)}_{\dot H_x^{5/4} \times \dot H_x^{1/4}} = \epsilon$$  
and assume
$$\norm{(u_0, u_1)}_{\dot H_x^1 \times L_x^2} = \delta_0 < 
\frac{\delta}{4 e^{4\tilde C_{s_0}^2 \epsilon^4} \tilde C_{s_0}^2}.$$ 

Fix a time $T$ with 
$$T \leq \frac1{2C_{s_0}\tilde C^2_{s_0} \epsilon^4},$$
and suppose that up to time $T$
that our norms have grown at most by the following constant factors:
$$\sup_{0 \leq t \leq T} \norm{(u(t), u'(t))}_{\dot H_x^{5/4} \times \dot H_x^{1/4}} 
\leq 2\sqrt{2} \tilde C_{s_0}\epsilon$$
and 
$$\sup_{0 \leq t \leq T} 
\norm{(u(t), u'(t))}_{\dot H_x^1 \times L_x^2} < \delta$$ 
We show that in fact these norms have not grown by more than half of these respective 
constants.

We see that for $0 \leq t \leq T$ the conditions for 
Theorem \ref{theorem:main} are met and we get for each 
$0 \leq s \leq s_0$ an energy 
$E^s$ which satisfies 
$$\abs{\partial_t E^s(t)} \leq C_{s_0} E^s(t) (E^{1/4}(t))^2$$
and 
$$\frac1{\tilde C_{s_0}} \norm{(u(t), u'(t))}^2_{\dot H_x^{s+1} \times \dot H_x^s} \leq 
E^s(t) \leq \tilde C_{s_0} \norm{u(t), u'(t)}^2_{\dot H_x^{s+1} \times \dot H_x^s}.$$

Setting $s = 1/4$, and using Gronwall's inequality we see that 
$$\abs{E^{1/4}(t)} \leq \frac{2}{\sqrt{1 - 2C_{s_0}T(E^{1/4}(0))^2}} E^{1/4}(0).$$
Thus, since $T \leq \frac1{4C_{k,N}\tilde C^2_{s_0} \epsilon^4}$ 
we have 
$$\abs{E^{1/4}(t)} \leq 2 \tilde C_{s_0} \epsilon^2.$$
In particular, this closes the continuity argument for the $\dot H_x^{5/4} \times \dot H_x^{1/4}$ 
norm:
$$\norm{(u(t), u(t))}_{\dot H_x^{5/4} \times \dot H_x^{1/4}} 
\leq \sqrt{2} \tilde C_{s_0} \epsilon.$$

To close the other inequality we use Gronwall's inequality again for $s = 0$
to get 
$$\abs{E^{0}(t)} \leq 2 e^{4\tilde C^2_{s_0} \epsilon^4}E^0(0).$$
This then closes the argument for the $\dot H_x^1 \times L_x^2$ norm:
$$\norm{(u(t), u'(t))}_{\dot H_x^1 \times L_x^2} \leq 2\tilde C_{k}^2 e^{4\tilde C^2_{k} \epsilon^4}
\norm{(u_0, u_1)}_{\dot H_x^1 \times L_x^2}
\leq 2e^{4\tilde C^2_{s_0} \epsilon^4}\tilde C_{s_0}^2 \delta_0 < \frac12 \delta$$

Thus, by the continuity in time of these norms given by the assumed local well-posedness 
theory, we have uniform bounds on our solutions $u$ on times less than 
$T =1/(4C_{s_0,N}\tilde C^2_{s_0} \epsilon^4)$:
$$\sup_{0 \leq t \leq T} \norm{(u(t), u'(t))}_{\dot H_x^{5/4} \times \dot H_x^{1/4}} 
\leq \sqrt{2} \tilde C_{s_0}\epsilon$$
and 
$$\sup_{0 \leq t \leq T} 
\norm{(u(t), u'(t))}_{\dot H_x^1 \times L_x^2} < \delta/2.$$ 

Finally, suppose for contradiction that the final time of existence $T^*$ for our 
solution $u$ is strictly less than $T$.

Then we satisfy the conditions of Theorem \ref{theorem:main}. Now using Gronwall's inequality 
for all $0\leq s \leq s_0$ we get the bounds 
$$\norm{(u(t), u'(t))}_{\dot H_x^{1+s} \times \dot H_x^s} \leq 
2\tilde C_{s_0}^2 e^{4\tilde C^2_{s_0} \epsilon^4}
\norm{(u_0, u_1)}_{\dot H_x^{1+s} \times \dot H_x^s}$$
To control the $L^2$ norm of $u$ we integrate $u'$:
$$\norm{u(t)}_{L^2} \leq T \norm{u'(t)}_{L^2} \leq \frac12 T\delta.$$
In particular, our assumed local well-posedness norm is a uniform constant (not depending 
on $T^*$) times the initial norm:
$$\norm{(u(T^*), u'(T^*))}_{H_x^{1+s_0} \times H_x^{s_0}} \lesssim
\norm{(u_0, u_1)}_{H_x^{1+s_0} \times H_x^{s_0}}. $$
Thus, the local well-posedness theory continues our solution some fixed amount of time not 
depending on $T^*$ until time $T$ which completes the proof.
\end{proof}

\subsection{Local Existence}
Next we turn our attention to Theorem \ref{thm:weak_solutions}.
Because the time of existence in Theorem \ref{theorem:enhanced} only depends on the $\dot H^{5/4}_x \times \dot H^{1/4}_x$
size of the initial data, solutions to \eqref{eq:kirchhoff} for approximate initial data in $H^{3/2}_x \times H^{1/2}_x$ 
exist on a uniform time interval. This allows the construction of solutions to \eqref{eq:kirchhoff} for 
$\dot H^{5/4}_x \times \dot H^{1/4}_x$ initial data 
as weak limits of these approximations. This method will not show that these solutions are in fact the strong 
limit of these approximations, which is why we call such solutions \emph{weak solutions}. 

However, we will need strong 
convergence in the $C^0_t(\dot H^1)$ topology to control the nonlinearity in \eqref{eq:kirchhoff}. To ensure strong 
convergence in a weaker topology we need some compactness. We get compactness through assuming compactly supported initial 
data and relying on finite speed of propagation for \eqref{eq:kirchhoff} with initial data in $H^{3/2}_x \times H^{1/2}_x$. 
We leave the proof of finite speed of propagation for Section \ref{app:finite_speed} of the appendix.
\begin{proposition}
    \label{prop:finite_speed}
    Let $u \in C^0_t([0,T], H^{3/2}_x) \times C^1_t([0,T], H^{1/2}_x)$ 
    solve \eqref{eq:kirchhoff} and suppose 
    that $u(0)$ and $u'(0)$ are supported in a ball of radius $R$ in space.
    Put 
    $$c = 1 + \sup_{0 \leq t \leq T} N(\norm{\nabla u(t)}^2_{L^2}).$$
    Then $u(t)$ is supported in a ball of radius $R + ct$.
\end{proposition}

Now we are in a position to restate and prove Theorem \ref{thm:weak_solutions}.

\begin{manualtheorem}{\ref{thm:weak_solutions}}
    Let $(u_0,u_1) \in H_x^{5/4} \times H_x^{1/4}$ be compactly supported with  
    $$\norm{(u_0, u_1)}_{\dot H_x^1 \times L_x^2} = \delta_0 \ll \delta$$
    where $\delta$ is the same small constant as in Theorem \ref{theorem:enhanced}.
    Denote the size of the initial data in $\dot H_x^{5/4} \times \dot H_x^{1/4}$ by 
    $$\epsilon = \norm{(u_0, u_1)}_{\dot H_x^{5/4} \times \dot H_x^{1/4}}.$$
    Then, for $T \sim 1/\epsilon^4$, as in Theorem \ref{theorem:enhanced}, 
    there exists a weak solution 
    $$u \in C_t^0([0,T], H_x^{5/4})\cap C^{1}_t([0,T], H_x^{1/4})$$
    which solves \eqref{eq:kirchhoff} for times in $[0,T)$ in the sense of distributions:

    For all $\phi \in C^\infty_c([0,T) \times U)$ (which vanish at $t=T$) we have 
    $$\int_{[0,T]\times \bb R^n} u \partial_t^2 \phi - u \Delta \phi  - u N(\norm{\nabla u}^2_{L^2})\Delta \phi \, dt dx
    = \int_{\bb R^n} u_0 \partial_t \phi \, dx - \int_{\bb R^n} u_1 \phi \, dx.$$
\end{manualtheorem}
\begin{proof}
    We piggyback off of the higher regularity local well-posedness result in $H_x^{3/2} 
    \times H_x^{1/2}$
    (say using the methods of  \cite{ifrim2022localwellposednessquasilinearproblems}) through applying 
    Theorem \ref{theorem:enhanced}. 

    Take a sequence of approximate compactly supported initial data 
    $(u^{(k)}_0, u^{(k)}_1) \in H_x^{3/2} \times H_x^{1/2}$ which converges to 
    $(u_0, u_1)$ in $H_x^{5/4} \times H_x^{1/4}.$ Since $(u_0, u_1)$ is size $\delta_0 << \delta$ in 
    $\dot H_x^1 \times L_x^2$ we make take these approximations to be smaller than $\delta_0$ in 
    $\dot H_x^1 \times L_x^2$.

    Then, applying Theorem \ref{theorem:enhanced} we get solutions 
    $$u^{(k)} \in C_t^0([0,T], H_x^{3/2}) \cap C_t^1([0,T], H_x^{1/2})$$
    with uniform estimates
    $$\norm{u^{(k)}}_{L_t^\infty([0,T], H_x^{5/4})} \lesssim T\delta + \epsilon$$
    and
    $$\norm{\partial_t u^{(k)}}_{L_t^\infty([0,T], H_x^{1/4})} \lesssim \delta + \epsilon.$$
    First, these uniform estimates immediately give us a subsequence which converges in the weak-$*$ 
    topology to a function 
    $$u \in L_t^\infty([0,T], H_x^{5/4}) \cap W_t^{1,\infty}([0,T], H_x^{1/4}).$$
    In particular, this means that for all $\phi \in C^\infty_c([0,T] \times \bb R^n)$ we have 
    $$\lim_{k \to \infty} \int_{[0,T] \times \bb R^n} u^{(k)} \phi \, dt dx \to 
    \int_{[0,T] \times \bb R^n} u \phi \, dt dx $$
    Now, let $R$ be large enough to contain the supports of $u^{(k)}(t)$ for all $k$ and 
    $0 \leq t \leq T$, which is guaranteed by Proposition \ref{prop:finite_speed}.
    We have the following
    compact embeddings by the Rellich–Kondrachov theorem
    $$H_x^{1/4}(B_R) \subset\subset H_x^1(B_R) \subset\subset H_x^{5/4}(B_R).$$
    By the Aubin-Lions lemma, we have the compact embedding 
    $$L_t^\infty([0,T], H_x^{5/4}(B_R)) \cap W_t^{1,\infty}([0,T], H_x^{1/4}(B_R)) \subset \subset C_t^0([0,T], H_x^1(B_R)).$$
    This means that there is a further subsequence which converges strongly 
    in the space $C_t^0([0,T], H_x^1(B_R))$, and that this limit must in fact be $u$. Further, 
    since all $u^{(k)}$ are supported in $B_R$ we also have convergence in 
    $C_t^0([0,T], H_x^1(\bb R^n)).$

    Finally, we will show that $u$ is indeed a solution to \eqref{eq:kirchhoff}. Since 
    each $u^{(k)}$ is a solution, it will suffice to show that
    \begin{align*}
    \lim_{k \to \infty} &\int_{[0,T]\times \bb R^n} u^{(k)} \partial_t^2 \phi - u^{(k)} \Delta \phi  - 
    u^{(k)} N(\norm{\nabla u^{(k)}}^2_{L^2})\Delta \phi \, dt dx
    - \int_{\bb R^n} u_0 \partial_t \phi \, dx + \int_{\bb R^n} u_1 \phi \, dx\\
    &= 
    \int_{[0,T]\times \bb R^n} u \partial_t^2 \phi - u \Delta \phi  
    - u N(\norm{\nabla u}^2_{L^2})\Delta \phi \, dt dx
    - \int_{\bb R^n} u_0 \partial_t \phi \, dx + \int_{\bb R^n} u_1 \phi \, dx.
    \end{align*}
    The linear terms follow from weak convergence of $u^{(k)} \rightharpoonup u$. 
    For the nonlinear term we write 
    \begin{align*}
    &\int_{[0,T]\times \bb R^n} 
    (u^{(k)}  N(\norm{\nabla u^{(k)}}^2_{L^2})
    - u  N(\norm{\nabla u}^2_{L^2})\Delta \phi \, dt dx\\
    &= \int_{[0,T]\times \bb R^n} (u^{(k)} 
    (N(\norm{\nabla u^{(k)}}^2_{L^2}) - N(\norm{\nabla u}^2_{L^2})) + (u^{(k)} - u) 
    N(\norm{\nabla u}^2_{L^2}))\Delta \phi \, dt dx
    \end{align*}
    The first term converges to $0$ since $\norm{u^{(k)}}^2_{\dot H_x^1}$ converges to 
    $\norm{u}^2_{\dot H_x^1}$ uniformly in $t$, since $N$ is continuous, and since $u^{(k)}$
    pairs with $\Delta \phi$. The second term 
    converges to $0$ by weak convergence since $N(\norm{\nabla u}^2_{L^2}))\Delta \phi$
    is a fixed function in the dual space of $u^{(k)} - u$.

Thus, we have a solution 
    $$u \in L_t^\infty([0,T], H_x^{5/4})\cap W^{1,\infty}_t([0,T], H_x^{1/4}).$$
To show that $u$ is in fact in the continuous space
    $C_t^0([0,T], H_x^{5/4})\cap C^{1}_t([0,T], H_x^{1/4}),$ we appeal to Theorem \ref{theorem:main}. Currently, 
    the proof of Theorem \ref{theorem:main} is written as if $E^s(t)$ is differentiable, which is immediately justified 
    if $u$ is the limit of strong solutions. At the cost of confusing the notation, however, Theorem \ref{theorem:main} can 
    be proved in the sense of distributions for any $u$ which solves \eqref{eq:kirchhoff} in the sense of distributions. That 
    is, the energy estimate in Theorem \ref{theorem:main} would be restated as follows.

For every $\phi \in C^\infty_c([0, T))$ we have 
$$\abs{\int_{[0,T]} E^s(t) \partial_t \phi(t) dt } \lesssim \abs{\int_{[0,T]} E^s(t)(E^{1/4}(t))^2 \phi(t) dt. }$$

But given such a distributional proof, it is the case that 
$E^{1/4} \in W^{1, \infty}_t([0,T))$. Since $u$ will be continuous in time measured in 
lower spatial regularity norms by the Aubin-Lions lemma, 
the continuity of $E^{1/4}$ implies the continuity in time of the norm
$$\norm{u(t)}_{L^\infty_t H^{5/4}_x \cap W^{1,\infty}_t H^{1/4}_x}.$$
Finally, if $u$ were not continuous at some time $t_0$, then we could recover a limit in 
a lower topology. For instance $u(t_0) = \lim_{t \to t_0} u(t)$
in $C^0_tH^1_x \cap C^1_tL^2_x$. Note that $u(t) - u(t_0) \to 0$ 
in $C^0_tH^1_x \cap C^1_tL^2_x$, but $u(t) - u(t_0) \not \to 0$ in $L^\infty_tH^{5/4}_x 
\cap W^{1,\infty}_tH^{1/4}_x$.
In particular, we would have 
$$\norm{u(t) - u(t_0)}_{L^\infty_tH^{5/4}_x 
\cap W^{1,\infty}_tH^{1/4}_x} \not \to 0,$$
which contradicts the continuity of the norm. Thus, we must in fact have that 
$$u \in C_t^0([0,T], H_x^{5/4})\cap C^{1,\infty}_t([0,T], H_x^{1/4}).$$
\end{proof}

\section{Remarks on the Linearized System} \label{sec:lin}

In this section we will discuss why we cannot use the same modified energy method as above
to infer cubic bounds 
on the linearized system which would be necessary for a local well-posedness result 
at the level of $H_x^{5/4} \times H_x^{1/4}$, at least using the Modified Energy method. 

Fundamentally, a uniqueness result compares solutions. Since comparing solutions is done through estimating differences of 
solutions, we require good estimates on the linearization of 
\eqref{eq:kirchhoff}. We do our analysis in the simplest setting where
$$N(\norm{\nabla u}_{L^2}^2) = \norm{\nabla u}_{L^2}^2.$$
The
linearization around a solution $u$ is 
\begin{equation}
\label{eq:linearized} w'' - \Delta w -  \norm{\nabla u}_{L^2_x}^2\Delta w  - 
2 \Delta u \int \nabla u(y) \nabla w(y)dy = 0.
\end{equation}
In order to get well-posedness in the Hadamard sense, we would like to have control of the
linearized energy for \eqref{eq:linearized}
but in a potentially weaker space $\dot H_x^{\sigma+1} \times \dot H_x^\sigma$ (see \cite{ifrim2022localwellposednessquasilinearproblems} for the full details).
To get a handle on the energy at the level of $\sigma$ we apply $\abs{\nabla}^\sigma$ to the equation 
and measure against $\abs{\nabla}^\sigma \bar w'$. 
Factoring out a time derivative where beneficial we get
\begin{align}
\nonumber    \partial_t&\left[\int \abs{\abs{\nabla}^\sigma w'}^2 + (1+\norm{\nabla u}_{L^2}^2)
\abs{\abs{\nabla}^{1+\sigma}w}^2\right]\\
\nonumber&= 
     \int \nabla u' \nabla u \int\abs{\abs{\nabla}^{1+\sigma}w} 
    - 2 \int \abs{\nabla}^{1+\sigma} w'\abs{\nabla}^{1+\sigma} u \int \nabla u \nabla w\\
\label{eq:sep}&= 
 \int_{\xi_1, \xi_2} \abs{\xi_1}^2\abs{\xi_2}^{2+2\sigma} 
\Re(\hat u'(\xi_1) \overline{\hat u}(\xi_1)) \abs{\hat w(\xi_2)}^2 \vec{d\xi}\\
\label{eq:mixed}&\qquad    - 2 \int_{\xi_1, \xi_2} \abs{\xi_1}^2\abs{\xi_2}^{2+2\sigma} 
\Re(\hat u(\xi_1) \overline{\hat w}(\xi_1))\Re(\hat u(\xi_2) \overline{\hat w}'(\xi_2)) 
\vec{d\xi}.
\end{align}

To hope to prove uniqueness of solutions $u \in C_t^0 (H_x^{5/4}) \cap C_t^1(H_x^{1/4}),$ we would need 
to bound the time derivative of the above energy of $w$ at the level of $\sigma$ in terms of 
itself and $\norm{u}_{C_t^0(H_x^{5/4}) \cap C_t^1(H_x^{1/4})}.$ To get this to happen, the term \eqref{eq:sep} 
requires an energy correction to better split up the derivatives $\abs{\xi_1}^2$. 
Controlling the term \eqref{eq:mixed} without a correction would require 
$$2 \leq \frac54 + 1 + \sigma\qquad \text{ and } \qquad 2+\sigma \leq \frac54 $$
which simplifies to 
$$-\frac14 \leq \sigma \leq -\frac34$$
which is of course impossible. Thus, both terms would require an energy correction to control.

It turns out that \eqref{eq:sep} can be controlled using a modified energy akin to the one 
in Section \ref{sec:simple}.
However, we will see that unfortunately \eqref{eq:mixed} cannot be corrected for any value of $\sigma.$ 
We will give two arguments of this. The first is heuristic and the second will be algebraic. 

\subsection{A Heuristic that (\ref{eq:mixed}) is Resonant}
We use the framework in Subsection \ref{ss:heuristics} and show that we are not guaranteed time 
oscillation in the linearized energy \eqref{eq:mixed}. First note that we cannot hope for cancellation
between the nonoscillatory parts of \eqref{eq:sep} and \eqref{eq:mixed} for the simple fact that
\eqref{eq:sep} is always oscillatory by the argument in Subsection \ref{ss:heuristics}, 
which we will not repeat.

For solutions $u$ to \eqref{eq:kirchhoff} we have the same heuristic decomposition as 
Subsection \ref{ss:heuristics}:
$$\hat u(t) \approx \sum c^+_\lambda e^{i\phi_\lambda} + c^-_\lambda e^{-i\phi_\lambda}$$
where 
$$\phi_\lambda(t) = \lambda t \sqrt{1 + \norm{u_0}^2} + \mathcal O(t^2).$$
(Again, we pretend that $c$ is constant when in reality it varies over frequency regions).

For $w$ we get a similar decomposition by taking the $\lambda$ spatial frequency component of 
\eqref{eq:linearized}.
$$
w_\lambda'' - \lambda^2(1 + \sum_\mu \mu^2 \abs{u_\mu}^2) w_\lambda 
-2 \lambda^2 u_\lambda \sum_\mu \mu^2 \Re(u_\mu \bar w_\mu) = 0.$$
Here the extra term mixes up the $w_\lambda$, but since the $\dot H_x^1$ norm of $u$ we expect to be 
small we can, to leading order, assume that the forcing between the $w_\lambda$ is negligible. 
We thus expect a decomposition like 
$$\hat w(t) \approx \sum d^+_\lambda e^{i\psi_\lambda} + d^-_\lambda e^{-i\psi_\lambda}$$
where heuristically 
$$\psi_\lambda(t) = \lambda t \sqrt{1 + \norm{u_0}^2 - 2(u_0)_\lambda^2} + \mathcal O(t^2).$$
Then we expand the product in \eqref{eq:mixed}.
\begin{align*}
    \int\abs{\xi_1}^{2} &\Re(\hat u(\xi_1) \bar{\hat w}(\xi_1)) d\xi_1
    \approx 
    \sum_{\lambda_1} \lambda_1^{2+2s} \Re(c^+_{\lambda_1} e^{i\phi_{\lambda_1}} 
        + c^-_{\lambda_1} e^{-i\phi_{\lambda_1}})(\bar d^+_{\lambda_1} e^{-i\psi_{\lambda_1}} 
        + \bar d^-_{\lambda_1} e^{i\psi_{\lambda_1}})\\
    &=
    \sum_{\lambda_1} \lambda_1^{2+2s} \Re((c^+_{\lambda_1}\bar d^+_{\lambda_1}  + 
    c^-_{\lambda_1}\bar d^-_{\lambda_1})e^{i(\phi_{\lambda_1} - \psi_{\lambda_1})}
    + (c^+_{\lambda_1}\bar d^-_{\lambda_1} + c^-_{\lambda_1}\bar d^+_{\lambda_1})e^{i(\phi_{\lambda_1} + \psi_{\lambda_1})})
\end{align*}
and 
\begin{align*}
    \int\abs{\xi_2}^{2} \Re(\hat u(\xi_2) &\bar{\hat w}'(\xi_2)) d\xi_2\\
    &\approx 
    \sum_{\lambda_2} \lambda_2^{2} \Re((c^+_{\lambda_2} e^{i\phi_{\lambda_2}} 
        + c^-_{\lambda_2} e^{-i\phi_{\lambda_2}})
        (-i\psi_{\lambda_2}'\bar d^+_{\lambda_2} e^{-i\psi_{\lambda_2}} 
        + i\bar d^-_{\lambda_2}\psi_{\lambda_2}' e^{i\psi_{\lambda_2}}))\\
    &=
    \sum_{\lambda_2} \lambda_2^2 \psi_{\lambda_2}'\Re((-i c^+_{\lambda_2}\bar d^+_{\lambda_2}  + 
    ic^-_{\lambda_2}\bar d^-_{\lambda_2})e^{i(\phi_{\lambda_2} - \psi_{\lambda_2})}
    + (c^+_{\lambda_2}\bar d^-_{\lambda_2} + c^-_{\lambda_2}\bar d^+_{\lambda_2})e^{i(\phi_{\lambda_2} + \psi_{\lambda_2})}).
\end{align*}
Again, we look at the case most likely to cause low frequency terms: $\lambda_1 = \lambda_2 \neq 0$.
Multiplying out we get products of $c$ and $\bar d$ which are not necessarily real and thus have no
hope of canceling in general since $c$ and $d$ are independent. Thus, we expect zero time frequency 
terms and cannot hope for extra cancellation. 

\subsection{Failure of Leading Order Algebraic Cancellation}
Here we show the above heuristic is true in the sense that the same algebraic techniques used in 
Section \ref{sec:simple} fail. 

Again notice that \eqref{eq:sep} and 
\eqref{eq:mixed} are independent. In particular, any term that could cancel with 
\eqref{eq:sep} will separate $u$ and $w$ by frequency ($u$ depends only on $\xi_1$ and 
$w$ depends only on $\xi_2$), but conversely any term that could cancel with \eqref{eq:mixed}
mix $u$ and $w$ in frequency (each frequency has a $u$ and $w$ component).

Thus, we cannot hope for any cancellation between \eqref{eq:sep} and \eqref{eq:mixed}, and must 
deal with them separately. It turns out that \eqref{eq:sep} can be corrected to a third 
order term in energies for $u$ and $w$ using a similar correction in Theorem 
\ref{theorem:main}. Unfortunately, \eqref{eq:mixed} is not amenable to cancellation.

Algebraically, we can show that no correction will fully cancel the 
leading order part of \eqref{eq:mixed}.
We write every possible leading order term that could interact with \eqref{eq:mixed}. Note that 
since we integrate over all $\xi_1, \xi_2$, we need to take into account cancellations 
that take place due to symmetry. To do this, we choose a single representation for each 
term. For instance, every term of the form  
$$c(\xi_1, \xi_2) u'(\xi_1) w(\xi_1) u(\xi_2)w'(\xi_2)$$
is equivalent to 
$$c(\xi_2, \xi_1)u(\xi_1) w'(\xi_1) u'(\xi_2)w(\xi_2)$$
so to reduce the extra degree of freedom we only use the second. Further, in terms 
that are already symmetric, we can symmetrize the coefficients, which we will denote 
by putting a superscript $sym$. Then any first order correction can be written 
\begin{align*}
    E^\sigma_{\xi_1\xi_2} 
&=
a^{sym}_{\xi_1\xi_2}u(\xi_1) w(\xi_1) u(\xi_2)w(\xi_2)\\
&+b^{sym}_{\xi_1\xi_2} u(\xi_1) w'(\xi_1) u(\xi_2)w'(\xi_2)\\
&+c_{\xi_1\xi_2} u(\xi_1) w'(\xi_1) u'(\xi_2)w(\xi_2)\\
&+d_{\xi_1\xi_2} u'(\xi_1) w'(\xi_1) u(\xi_2)w(\xi_2)\\
&+e^{sym}_{\xi_1\xi_2} u'(\xi_1) w(\xi_1) u'(\xi_2)w(\xi_2)\\
&+f^{sym}_{\xi_1\xi_2} u'(\xi_1) w'(\xi_1) u'(\xi_2)w'(\xi_2).
\end{align*}
Asking for the time derivative of this cancels to leading order with \eqref{eq:mixed} boils 
down to the following system of equations
\begin{align}
\label{eq:linear1} 2a^{sym}_{\xi_1\xi_2} - 2b^{sym}_{\xi_1\xi_2}\xi_2^2 
- c_{\xi_1\xi_2}\xi_2^2 - d_{\xi_1\xi_2}\xi_1^2 &= \xi_1^{2+2\sigma} \xi_2^2\\ \label{eq:linear2} 2a^{sym}_{\xi_1\xi_2}  - c_{\xi_2\xi_1}\xi_2^2 
- d_{\xi_1\xi_2}\xi_1^2 - 2e^{sym}_{\xi_1\xi_2}\xi_2^2 &= 0\\
\label{eq:linear3} c_{\xi_1\xi_2} + d_{\xi_1\xi_2} 
+ 2e^{sym}_{\xi_1\xi_2} - 2f^{sym}_{\xi_1\xi_2}\xi_2^2 &= 0\\
\label{eq:linear4} 2b^{sym}_{\xi_1\xi_2} + c_{\xi_2\xi_1} 
+ d_{\xi_1\xi_2} - 2f^{sym}_{\xi_1\xi_2}\xi_2^2 &= 0.
\end{align}
Taking the asymmetric parts of \eqref{eq:linear3} and \eqref{eq:linear4} shows 
$$c_{\xi_1\xi_2} = c_{\xi_2\xi_1},$$
and the symmetric parts of \eqref{eq:linear3} and \eqref{eq:linear4} shows 
$$b^{sym}_{\xi_1\xi_2} = e^{sym}_{\xi_1\xi_2}.$$
These together shows that \eqref{eq:linear1} and \eqref{eq:linear2} are in fact equal, 
so that this can only cancel when 
$$0 =  \xi_1^{2+2\sigma} \xi_2^2$$
which of course is not true for any nonzero frequencies independent of the choice of $\sigma$.

\section{Appendix}\label{sec:appendix}
\subsection{Proof of Proposition \ref{prop:coefficients}}
\label{app:coeff}
\begin{proof}
In the context of this lemma we were working with the model nonlinearity 
$$\hat u''(\xi_1) = - \left(1 + A \int_{\xi_2} \abs{\xi_2}^2 \abs{\hat u(\xi_2)}^2 d\xi_2\right)
\abs{\xi_1}^2 \hat u(\xi_1)$$
We want to show that 
\begin{align*}
    \partial_t \int_{\xi_1, \xi_2} A E^s_{\xi_1\xi_2} \vec{d\xi}&= 
    -\int_{\xi_1, \xi_2} A\abs{\xi_1}^{2+2s}\abs{\xi_2}^2 \abs{\hat u(\xi_1)}^2 
\Re(\hat u(\xi_2) \bar {\hat u}(\xi_2)) \vec{d\xi}\\
&- \int_{\xi_1,\xi_2,\xi_3} \frac12 A^2 (\xi_1^{2s} - \xi_2^{2s})\xi_1^2\xi_2^2\xi_3^2
\abs{\hat u(\xi_1)}^2 \Re(\hat u(\xi_2) \bar {\hat u}'(\xi_2)) \abs{\hat u(\xi_3)}^2\vec{d\xi}
\end{align*}
where $E^s$ is given in Definition \ref{def:coefficients} by
$$E^s_{\xi_1\xi_2} = a_{\xi_1\xi_2} \abs{\hat u(\xi_1)}^2 \abs{\hat u(\xi_2)}^2
+ b_{\xi_1\xi_2} \abs{\hat u(\xi_1)}^2\abs{\hat u'(\xi_2)}^2
+ c_{\xi_1\xi_2} \Re(\hat u(\xi_1) \bar {\hat u}'(\xi_1) \hat u(\xi_2) \bar {\hat u}'(\xi_2)) $$
$$a_{\xi_1\xi_2} = -\frac18 \abs{\xi_1}^2 \abs{\xi_2}^2(\abs{\xi_1}^{2s} + \abs{\xi_2}^{2s})$$
    $$b_{\xi_1\xi_2} = -c_{\xi_1\xi_2} = 
    -\frac14 \abs{\xi_1}^2 \abs{\xi_2}^2 \frac{\abs{\xi_1}^{2s} - \abs{\xi_2}^{2s}}
    {\abs{\xi_1}^2 - \abs{\xi_2}^2}$$
Let us work term by term noting that $a$, $b$, and $c$ are symmetric in $\xi_1$ and $\xi_2$.
\begin{align*}
    \partial_t &a_{\xi_1\xi_2} \abs{\hat u(\xi_1)}^2\abs{\hat u(\xi_2)}^2\\
               &= -\frac14 \abs{\xi_1}^2\abs{\xi_2}^2(\abs{\xi_1}^{2s} + 
    \abs{\xi_2}^{2s})\left(\Re(\hat u(\xi_1) \bar {\hat u}'(\xi_1))\abs{\hat u(\xi_2)}^2
    + \abs{\hat u(\xi_1)}^2\Re(\hat u(\xi_2) \bar {\hat u}'(\xi_2))\right)
\end{align*}
    So that 
\begin{align*}
    \partial_t &\int_{\xi_1, \xi_2} a_{\xi_1\xi_2} \abs{\hat u(\xi_1)}^2\abs{\hat u(\xi_2)}^2
    \vec{d\xi}\\
    &= -\int_{\xi_1, \xi_2} \frac12 \abs{\xi_1}^2\abs{\xi_2}^2(\abs{\xi_1}^{2s} + 
    \abs{\xi_2}^{2s})\abs{\hat u(\xi_1)}^2\Re(\hat u(\xi_2) \bar {\hat u}'(\xi_2)) \vec{d\xi}
\end{align*}
Similarly for $b$ we have 

    \begin{align*}
        \partial_t &b_{\xi_1\xi_2} \abs{\hat u(\xi_1)}^2\abs{\hat u'(\xi_2)}^2 
=-\frac12 \abs{\xi_1}^2 \abs{\xi_2}^2 \frac{\abs{\xi_1}^{2s} - \abs{\xi_2}^{2s}}
{\abs{\xi_1}^2 - \abs{\xi_2}^2}
    \Bigg(\Re(\hat u(\xi_1) \bar {\hat u}'(\xi_1))\abs{\hat u'(\xi_2)}^2\\
&- \abs{\xi_2}^2\abs{\hat u(\xi_1)}^2\Re(\hat u(\xi_2) \bar {\hat u}'(\xi_2))
- \int_{\xi_3} A\abs{\xi_2}^2\abs{\xi_3}^2 \abs{\hat u(\xi_1)}^2\Re(\hat u(\xi_2) \bar {\hat u}'(\xi_2))
    \abs{\hat u(\xi_3)}^2 d\xi_3
\Bigg).
    \end{align*}
For the $c$ term it will be helpful to use the symmetry from the beginning to avoid clutter:
\begin{align*}
    \partial_t &\int_{\xi_1,\xi_2}c_{\xi_1\xi_2} \Re(\hat u(\xi_1) \bar {\hat u}'(\xi_1))
\Re(\hat u(\xi_2) \bar {\hat u}'(\xi_2)) \vec{d\xi}
= \int_{\xi_1,\xi_2}\frac12 \abs{\xi_1}^2 \abs{\xi_2}^2 \frac{\abs{\xi_1}^{2s} - \abs{\xi_2}^{2s}}
{\abs{\xi_1}^2 - \abs{\xi_2}^2}\\
               &\qquad \Bigg(\Re(\hat u(\xi_1) \bar {\hat u}'(\xi_1))\abs{\hat u'(\xi_2)}^2
-\abs{\xi_1}^2\abs{\hat u(\xi_1)}^2\Re(\hat u(\xi_2) \bar {\hat u}'(\xi_2))\\
               &\qquad- \int_{\xi_3} A\abs{\xi_1}^2\abs{\xi_3}^2 
\abs{\hat u(\xi_1)}^2\Re(\hat u(\xi_2) \bar {\hat u}'(\xi_2))\abs{\hat u(\xi_3)}^2d\xi_3\Bigg) 
\vec{d\xi}
\end{align*}
Combining like terms we get 
\begin{align*}
    \partial_t &\int_{\xi_1, \xi_2} E^s_{\xi_1\xi_2}\vec{d\xi} = \\
               &\int_{\xi_1,\xi_2}\frac12 \Bigg(-(\abs{\xi_1}^{2s} + \abs{\xi_2}^{2s}) 
    + \abs{\xi_2}^2 \frac{\abs{\xi_1}^{2s} - \abs{\xi_2}^{2s}}
{\abs{\xi_1}^2 - \abs{\xi_2}^2}
    - \abs{\xi_1}^2 \frac{\abs{\xi_1}^{2s} - \abs{\xi_2}^{2s}}
{\abs{\xi_1}^2 - \abs{\xi_2}^2}
\Bigg)\\
               &\qquad\qquad \abs{\xi_1}^2\abs{\xi_2}^2\abs{\hat u(\xi_1)}^2 \Re(\hat u(\xi_2)\bar {\hat u}'(\xi_2))\vec{d\xi} \\
&+\int_{\xi_1,\xi_2}\Bigg(-\frac12 \frac{\abs{\xi_1}^{2s} - \abs{\xi_2}^{2s}}
{\abs{\xi_1}^2 - \abs{\xi_2}^2}
+\frac12 \frac{\abs{\xi_1}^{2s} - \abs{\xi_2}^{2s}}
{\abs{\xi_1}^2 - \abs{\xi_2}^2}
\Bigg)
\abs{\xi_1}^2\abs{\xi_2}^2\Re(\hat u(\xi_1) \bar {\hat u}'(\xi_1))\abs{\hat u'(\xi_2)}^2\vec{d\xi} \\
&+\int_{\xi_1,\xi_2,\xi_3}\Bigg(\frac12 \abs{\xi_2}^2\frac{\abs{\xi_1}^{2s} - \abs{\xi_2}^{2s}}
{\abs{\xi_1}^2 - \abs{\xi_2}^2}
-\frac12 \abs{\xi_1}^2\frac{\abs{\xi_1}^{2s} - \abs{\xi_2}^{2s}}
{\abs{\xi_1}^2 - \abs{\xi_2}^2}
\Bigg)\\
&\qquad \qquad A\abs{\xi_3}^2\abs{\hat u(\xi_1)}^2
\Re(\hat u(\xi_2) \bar {\hat u}'(\xi_2))\abs{\hat u(\xi_3)}^2\vec{d\xi} 
\end{align*}
After simplifying and multiplying through by $A$ we get the desired equality
\begin{align*}
    \partial_t \int_{\xi_1, \xi_2} A E^s_{\xi_1\xi_2} \vec{d\xi} &= 
    -\int_{\xi_1, \xi_2} A\abs{\xi_1}^{2+2s}\abs{\xi_2}^2 \abs{\hat u(\xi_1)}^2 
\Re(\hat u(\xi_2) \bar {\hat u}(\xi_2))\vec{d\xi} \\
&- \int_{\xi_1,\xi_2,\xi_3} \frac12 A^2 (\xi_1^{2s} - \xi_2^{2s})\xi_1^2\xi_2^2\xi_3^2
\abs{\hat u(\xi_1)}^2 \Re(\hat u(\xi_2) \bar {\hat u}'(\xi_2)) \abs{\hat u(\xi_3)}^2\vec{d\xi} 
\end{align*}
\end{proof}
\subsection{Proof of Lemma \ref{lem:kernel}}
\label{app:kernel}
\begin{proof}
    We start with a proof of Lemma \ref{lem:kernel} when $s \geq 0$.
    
    When $\abs{\xi_1} \leq \abs{\xi_2}$ we factor 
    $$ \frac{\abs{\xi_1}^{2s} - \abs{\xi_2}^{2s}}{\abs{\xi_1}^2 - \abs{\xi_2}^2}
    \leq \frac{\abs{\xi_2}^{2s}}{\abs{\xi_2}^2}
    \frac{\left(\frac{\abs{\xi_1}}{\abs{\xi_2}}\right)^{2s}- 1}
    {\left(\frac{\abs{\xi_1}}{\abs{\xi_2}}\right)^2 - 1}.$$
    Thus, we will be done when we show 
    $$\abs{\frac{1- x^{2s}}{1- x^2}} \leq 1 + s$$
    for $0 \leq x \leq 1$.

    First, note that our function is nonnegative since it is continuous in $x$ and $s$ 
    and is the division of positive numbers
    when $0 \leq  x < 1$, so we don't need to worry about the absolute value.
    
    When $0 \leq s \leq 1$ we have that $x^2 \leq x^{2s}$ since $0 \leq x \leq 1$ and so 
    $$1 - x^{2s} \leq 1 - x^2 \implies \frac{1 - x^{2s}}{1 - x^2} \leq 1$$

    When $s > 1$ we can use differentiability in $s$ in the following way:
    First, when $s$ in an integer we have 
    $$\frac{1- x^{2s}}{1- x^2} = 1 + x^2 + \cdots + x^{2s - 2} \leq s$$
    since there are $s$ terms of at most size $1$. 

    Now fix $0 < x < 1$. For this $x$ we see that the function 
    $$f(s) = \frac{1- x^{2s}}{1- x^2} $$
    is differentiable in $s$ and its derivative is 
    $$f'(s) = -\frac2{1-x^2} \log(x) x^{2s}$$
    which is strictly positive for our fixed $x$. Thus, if there was an $s$ for which 
    our inequality fails we would have 
    $$f(s) > s+1 > \lceil s\rceil \geq f(\lceil s\rceil)$$
    which violates the mean value theorem. The endpoints in $x$ follow from continuity in $x$.

    ~

    It turns out that the $s \leq 0$ case follows from the $s \geq 0$ case. We again set 
    $$x = \frac{\abs{\xi_1}}{\abs{\xi_2}}$$
    so that our desired inequality becomes, for $0 \leq x \leq 1$ and $s \leq 0$
    $$\frac{x^{2s} - 1}{1- x^2} \leq (1 + \abs{s})x^{2s}.$$
    (Note that since $s \leq 0$ the left hand side of the equation is nonnegative.) 
    Dividing by $x^{2s}$ we would like to show 
    $$\frac{1 - x^{-2s}}{1- x^2} \leq (1 + \abs{s}).$$
    which is exactly what we showed above, since $-2s \geq 0$
\end{proof}

\subsection{Proof of Lemma \ref{lem:f}}
\label{app:f}
\begin{proof}
The bound on $\abs{F}$ is clear since the integral is nonnegative and bounded by 
$\norm{u(t)}_{\dot H_x^1}.$
For the time derivative we calculate 
$$\partial_t F = -3A \left(1 + A \int_{\abs{\xi_3} \leq r} \abs{\xi_3}^2 
\abs{\hat u(\xi_3)}^2d\xi_3\right)^{-5/2}\int_{\abs{\xi_3} \leq r} \abs{\xi_3}^2 \Re(\hat u(\xi_3) 
\bar {\hat u}'(\xi_3))d\xi_3$$
which similarly gives the desired bound.

For the identity, it would follow by taking a derivative in $r$ and following the 
discussion preceding Lemma \ref{lem:f} if $F$ was differentiable in $r$, which may not be 
the case if $u$ is rough. The computation is rigorously justified by a dominated smooth 
approximation to $u$.
\end{proof}
\subsection{Proof of Proposition \ref{prop:finite_speed}}
\label{app:finite_speed}
\begin{proof}
We consider the $\dot H^1 \times L^2$ unmodified energy in the complement of the ball of 
radius $R + ct$ which we denote 
$$E(t) = \int_{B^c_{R + ct}} \frac12 \abs{u'}^2 + \frac12(1 + N(\norm{\nabla u}^2_{L^2})) \abs{\nabla u}^2dx.$$
Then we see that 
\begin{align*}
\partial_t E(t) 
&= -c\int_{S_{R + ct}} \frac12 \abs{u'}^2 + \frac12(1 + N(\norm{\nabla u}^2_{L^2})) \abs{\nabla u}^2dx\\
&\qquad+ N'(\norm{\nabla u}^2_{L^2})\int_{B_{R + ct}} \abs{\nabla u}^2 dx\int_{\bb R^n} 
\Re(\nabla u \nabla \bar u') dx \\
&\qquad + \int_{B_{R + ct}} \Re(u''\bar u' + (1 + N(\norm{\nabla u}^2_{L^2})) \nabla u \nabla 
    \bar u') dx \\
&= -\int_{S_{R + ct}} \frac12 c\abs{u'}^2 + (1 + N(\norm{\nabla u})) \Re(\bar u' n \cdot \nabla u) 
+ c\frac12(1 + N(\norm{\nabla u}^2_{L^2})) 
    \abs{\nabla u}^2dx \\
&\qquad+ N'(\norm{\nabla u}^2_{L^2}) \int_{B_{R + ct}} \abs{\nabla u}^2 dx\int_{\bb R^n} 
\Re(\nabla u \nabla \bar u') dx \\
&\leq -\int_{S_{R + ct}} \frac12 c\abs{u'}^2 - c \abs{\bar u' \nabla u} 
+ c\frac12(1 + N(\norm{\nabla u}^2_{L^2})) 
    \abs{\nabla u}^2dx \\
&\qquad+ N'(\norm{\nabla u}^2_{L^2}) E(t) \norm{u(t)}_{\dot H^{3/2} \times \dot H^{1/2}}\\
&\leq -\frac12 c\int_{S_{R + ct}} (\abs{u'} - \abs{\nabla u})^2
+ N'(\norm{\nabla u}^2_{L^2}) E(t) \norm{u(t)}_{\dot H^{3/2} \times \dot H^{1/2}}\\
&\leq  N'(\norm{\nabla u}^2_{L^2}) E(t) \norm{u(t)}^2_{\dot H^{3/2} \times \dot H^{1/2}}
\end{align*}
Thus, by Gronwall's inequality we have 
$$E(t) \leq E(0)\exp({\sup_{0 \leq t \leq T} N'(\norm{\nabla u(t)}^2_{L^2})
\norm{u(t)}^2_{\dot H^{3/2} \times \dot H^{1/2}}})$$
Since $E(0) = 0$ by assumption we have that $E(t) = 0$ for all $0 \leq t \leq T$ which 
in turn forces the support of $u(t)$ to be contained in the ball of radius $R+ct$.
\end{proof}

\bibliographystyle{IEEEtran}
\bibliography{refs}

\begin{thebibliography}{10}
\providecommand{\url}[1]{#1}
\csname url@samestyle\endcsname
\providecommand{\newblock}{\relax}
\providecommand{\bibinfo}[2]{#2}
\providecommand{\BIBentrySTDinterwordspacing}{\spaceskip=0pt\relax}
\providecommand{\BIBentryALTinterwordstretchfactor}{4}
\providecommand{\BIBentryALTinterwordspacing}{\spaceskip=\fontdimen2\font plus
\BIBentryALTinterwordstretchfactor\fontdimen3\font minus
  \fontdimen4\font\relax}
\providecommand{\BIBforeignlanguage}[2]{{%
\expandafter\ifx\csname l@#1\endcsname\relax
\typeout{** WARNING: IEEEtran.bst: No hyphenation pattern has been}%
\typeout{** loaded for the language `#1'. Using the pattern for}%
\typeout{** the default language instead.}%
\else
\language=\csname l@#1\endcsname
\fi
#2}}
\providecommand{\BIBdecl}{\relax}
\BIBdecl

\bibitem{kirchhoff1897vorlesungen}
G.~Kirchhoff, \emph{Vorlesungen {\"u}ber mechanik}.\hskip 1em plus 0.5em minus
  0.4em\relax BG Teubner, 1897, vol.~1.

\bibitem{arosio1991mildly}
A.~Arosio and S.~Garavaldi, ``On the mildly degenerate kirchhoff string,''
  \emph{Mathematical Methods in the Applied Sciences}, vol.~14, no.~3, pp.
  177--195, 1991.

\bibitem{arosio1996well}
A.~Arosio and S.~Panizzi, ``On the well-posedness of the kirchhoff string,''
  \emph{Transactions of the American Mathematical Society}, vol. 348, no.~1,
  pp. 305--330, 1996.

\bibitem{kato1975cauchy}
T.~Kato, ``The cauchy problem for quasi-linear symmetric hyperbolic systems,''
  \emph{Archive for Rational Mechanics and Analysis}, vol.~58, no.~3, pp.
  181--205, 1975.

\bibitem{ifrim2022localwellposednessquasilinearproblems}
\BIBentryALTinterwordspacing
M.~Ifrim and D.~Tataru, ``Local well-posedness for quasilinear problems: a
  primer,'' 2022. [Online]. Available: \url{https://arxiv.org/abs/2008.05684}
\BIBentrySTDinterwordspacing

\bibitem{hadamard1923lectures}
J.~Hadamard, \emph{Lectures on Cauchy's Problem in Linear Partial Differential
  Equations}.\hskip 1em plus 0.5em minus 0.4em\relax Yale University Press,
  1923, vol.~15.

\bibitem{Baldi_2020}
\BIBentryALTinterwordspacing
P.~Baldi and E.~Haus, ``On the existence time for the kirchhoff equation with
  periodic boundary conditions,'' \emph{Nonlinearity}, vol.~33, no.~1, p. 196,
  nov 2019. [Online]. Available:
  \url{https://dx.doi.org/10.1088/1361-6544/ab4c7b}
\BIBentrySTDinterwordspacing

\bibitem{Baldi_longer}
\BIBentryALTinterwordspacing
------, ``Longer lifespan for many solutions of the kirchhoff equation,''
  \emph{SIAM Journal on Mathematical Analysis}, vol.~54, no.~1, pp. 306--342,
  2022. [Online]. Available: \url{https://doi.org/10.1137/20M1351515}
\BIBentrySTDinterwordspacing

\bibitem{Shatah85}
\BIBentryALTinterwordspacing
J.~Shatah, ``Normal forms and quadratic nonlinear klein-gordon equations,''
  \emph{Communications on Pure and Applied Mathematics}, vol.~38, no.~5, pp.
  685--696, 1985. [Online]. Available:
  \url{https://onlinelibrary.wiley.com/doi/abs/10.1002/cpa.3160380516}
\BIBentrySTDinterwordspacing

\bibitem{ono1997global}
K.~Ono, ``Global existence, decay, and blowup of solutions for some mildly
  degenerate nonlinear kirchhoff strings,'' \emph{Journal of differential
  equations}, vol. 137, no.~2, pp. 273--301, 1997.

\bibitem{ONO19974449}
\BIBentryALTinterwordspacing
------, ``Blowing up and global existence of solutions for some degenerate
  nonlinear wave equations with some dissipation,'' \emph{Nonlinear Analysis:
  Theory, Methods \& Applications}, vol.~30, no.~7, pp. 4449--4457, 1997,
  proceedings of the Second World Congress of Nonlinear Analysts. [Online].
  Available:
  \url{https://www.sciencedirect.com/science/article/pii/S0362546X97001831}
\BIBentrySTDinterwordspacing

\bibitem{d1992global}
P.~D'ancona and S.~Spagnolo, ``Global solvability for the degenerate kirchhoff
  equation with real analytic data,'' \emph{Inventiones mathematicae}, vol.
  108, no.~1, pp. 247--262, 1992.

\bibitem{nishihara1984global}
K.~Nishihara, ``On a global solution of some quasilinear hyperbolic equation,''
  \emph{Tokyo journal of mathematics}, vol.~7, no.~2, pp. 437--459, 1984.

\bibitem{HITW13}
J.~Hunter, M.~Ifrim, D.~Tataru, and T.~Wong, ``Long time solutions for a
  burgers-hilbert equation via a modified energy method,'' \emph{Proceedings of
  the American Mathematical Society}, vol. 143, 01 2013.

\bibitem{hunter2014dimensionalwaterwavesholomorphic}
\BIBentryALTinterwordspacing
J.~Hunter, M.~Ifrim, and D.~Tataru, ``Two dimensional water waves in
  holomorphic coordinates,'' 2014. [Online]. Available:
  \url{https://arxiv.org/abs/1401.1252}
\BIBentrySTDinterwordspacing

\bibitem{ifrim2017lifespansmalldatasolutions}
\BIBentryALTinterwordspacing
M.~Ifrim and D.~Tataru, ``The lifespan of small data solutions in two
  dimensional capillary water waves,'' 2017. [Online]. Available:
  \url{https://arxiv.org/abs/1406.5471}
\BIBentrySTDinterwordspacing

\bibitem{ai2023dimensionalgravitywaveslow}
\BIBentryALTinterwordspacing
A.~Ai, M.~Ifrim, and D.~Tataru, ``Two dimensional gravity waves at low
  regularity i: Energy estimates,'' 2023. [Online]. Available:
  \url{https://arxiv.org/abs/1910.05323}
\BIBentrySTDinterwordspacing

\end{thebibliography}
\end{document}